\rmfamily\documentclass[english,11pt]{smfartVScomptage}

\usepackage{smfenum,smfthm, amsmath,amssymb,amscd,mathrsfs,euscript,color}
\usepackage{stmaryrd,mathabx,upgreek} 
\usepackage{enumitem} 
\usepackage[latin1]{inputenc}
\usepackage[T1]{fontenc}
\usepackage{xcolor}

\usepackage{smfenum,smfthm, amsmath,amssymb,amscd,mathrsfs,euscript,color}
\usepackage{stmaryrd,mathabx,upgreek} 
\usepackage{enumitem} 
\usepackage[latin1]{inputenc}
\usepackage[T1]{fontenc}
\usepackage{xcolor}
\input{xypic}
\usepackage{hyperref}

\numberwithin{equation}{section} 
\def\AA{\mathfrak{A}}

\def\CC{\mathbb{C}}

\def\G{{ G}}
\def\H{{ H}}

\def\L{{\rm L}}
\def\M{{\rm M}}

\def\O{{\rm O}}

\def\R{{ R}}

\def\Cc{\EuScript{C}}

\def\Gg{\mathcal{G}}
\def\Hh{\mathcal{H}}

\def\Zz{\mathcal{Z}}

\def\La{{\it\Lambda}}

\def\a{\alpha} 

\def\d{\delta}

\def\g{\gamma}

\def\k{\boldsymbol{k}}

\def\o{\mathfrak{o}}

\def\kk{\boldsymbol{k}}
\def\ll{\mathrm{l}}

\def\({\left(}
\def\){\right)}
\def\>{\geqslant}
\def\<{\leqslant}

\def\Hom{\operatorname{Hom}}
\def\End{\operatorname{End}}

\def\Mat{\operatorname{M}}
\def\GL{\operatorname{GL}}
\def\Gal{\operatorname{Gal}}
\def\tr{\operatorname{tr}}

\def\dim{\operatorname{dim}}

\def\diag{\operatorname{diag}}

\def\st{\operatorname{st}}

\def\oo{\EuScript{O}}
\def\Id{\mathrm{Id}}

\def\1{\mathbf{1}}

\def\Hb{\mathrm{Hom}_{\mathrm{spec}}}

\def\Fl{\overline{\mathbb{F}_\ell}}

\def\Zl{\overline{\mathbb{Z}_\ell}}

\def\ll{\boldsymbol{l}}

\def\Ker{{\rm Ker}}

\def\Mat{\boldsymbol{{\sf M}}}

\def\Cc{\EuScript{C}}

\title[Godement--Jacquet~gamma factors of distinguished representations]
{Godement--Jacquet~gamma factors of distinguished representations of~$\GL_n(\mathbb{F}_q)$} 
\author{Robert Kurinczuk}
\address{School of Mathematics and Statistics, University of Sheffield,
Sheffield, S3 7RH, United Kingdom}
\email{robkurinczuk@gmail.com}
\author{Nadir Matringe}
\address{Shanghai Institute for Mathematics and Interdisciplinary Sciences, Block A, International Innovation Plaza, No. 657 Songhu Road, Yangpu District, Shanghai, China}
\email{nadirmatringe@outlook.fr}
\author{Vincent S\'echerre} 
\address{Laboratoire de Math\'emati\-ques de Versailles, UVSQ, CNRS, Universit\'e Paris-Saclay, 78035, Versailles, France}
\email{vincent.secherre@uvsq.fr}

\begin{document}

\maketitle

{\hfill\textcolor{green}{\bf \today}}

\begin{abstract} Let~$\k$ be a finite field of characteristic~$p$. In the 1960s, Kondo attached non-abelian Gauss sums to irreducible~$\mathbb{C}$-representations of~$\GL_n(\k)$, and computed them in terms of Green parameters. On the other hand, the Godement-Jacquet functional equation in which they occur was established by Macdonald in the 1980s. We first revisit Macdonald's and Kondo's results with a different perspective, in the process of generalizing their constructions to representations with coefficients in~$\mathbb{Z}[\sqrt{p}^{-1},\mu_p]$-algebras. Then, when~$p$ is odd and $R$  is an algebraically closed field of characteristic different to~$p$, our main result shows that the Godement-Jacquet gamma factor of a cuspidal~irreducible~$R$-representation, which is distinguished with respect to the subgroup fixed by a Galois or an inner involution, coincides with the sign of the associated period under the normalizer of this subgroup.  % this result is used in a companion paper to classify the~$\ell$-modular cuspidal representations of non-archimedean $\GL_n$ distinguished by a Galois involution. 
Finally, we compute the gamma factors of these distinguished representations in terms of Green's and James' parametrizations of irreducible cuspidal~$R$-representations. 
\end{abstract}

\setcounter{tocdepth}{1}
\tableofcontents

\section{Introduction}
\subsection{Background}
For~$\GL_n$ over a non-archimedean local field, root numbers and central gamma function values of (pairs of) distinguished smooth irreducible complex representations have been extensively studied.  
%We refer to the far from being exhaustive list \cite{FQ}, \cite{Hak},  \cite{Ok}, \cite{Anand}, \cite{HO}, \cite{MO}, \cite{Delboy}, \cite{PR}, \cite{X}, \cite{SecENS}. 

\subsection{Finite fields gamma factors and their functional equation}\label{sec KoMac}
Now let~$\k$ be a finite field of characteristic~$p$ and cardinality~$q$.  In this setting, finite field avatars of gamma factors for complex representations of~$\GL_n(\k)$ have been considered in several works, for example \cite{PSBook}, \cite{Roddity}, \cite{Nien}, \cite{SZ}, and \cite{bakeberg2025modellgammafactors} in the setting of representations with coefficients in a $\mathbb{Z}[\sqrt{p}^{-1},\mu_p]$-algebra. For cuspidal representations, finite field gamma factors often coincide with associated depth zero~$p$-adic counterparts; this has been proved in \cite{Ye}, and independently in \cite{Nienetal}, for Rankin--Selberg gamma factors of cuspidal complex representations in unequal rank (and in \cite{Ye} in equal rank whenever the non-archimedean $L$-factors are trivial), and the same holds for the Godement--Jacquet gamma factors under consideration here: either following from Macdonald's results in \cite{Macdonald} and known properties of the local Langlands correspondence, or by a simple argument on lifting of finite field matrix coefficients. In particular, many results for finite field gamma factors of complex representations can be deduced from known properties of $p$-adic gamma factors. However for representations with coefficients in more general rings, it can be useful to work in the other direction due the lack of analytic tools. %This is in fact the main motivation for this paper, and we work at the level of finite fields throughout this whole work.  

%The two main and original references for 
The gamma factors that we consider here were introduced in the complex setting by \cite{Kondo} and studied by \cite{Macdonald}: For~$\pi$ an irreducible~$\mathbb{C}$-representation of~$\GL_n(\k)$ and~$\psi:\k\rightarrow \mathbb{C}^\times$ a non-trivial character, Kondo \cite{Kondo} defined the non-abelian Gauss sum~$\gamma(\pi,\psi)$, and computed it in terms of the Green parameter of~$\pi$.  It was then proved by Springer \cite{Springer} in the cuspidal case, and Macdonald \cite{Macdonald} whenever the cuspidal support of~$\pi$ has no trivial factor, that such Gauss sums are the proportionality constants of finite field analogues of Godement--Jacquet functional equations \cite{GJ}.  It is worth mentioning that a recent and geometric proof of Kondo's formula of~$\gamma(\pi,\psi)$ was given by Braverman--Kazhdan in \cite[Section 9]{BK}, initiating much research on the topic of non-abelian Gauss sums for finite reductive groups.

\subsection{Our results on Godement--Jacquet gamma factors} 
In Section \ref{GJFEs}, we review the definition of finite field Godement--Jacquet gamma factors of~$\GL_n(\k)$ and extend this definition from the setting of complex representations (on~$\mathbb{C}$-vector spaces) to the setting of representations on~$R$-modules for quite general $R$.  More precisely we introduce Godement--Jacquet gamma factors under the minimal hypothesis that~$R$ is a~$\mathbb{Z}[\mu_p,\sqrt{p}^{-1}]$-algebra, this hypothesis allowing for a normalized Fourier transform as in Section \ref{fouriersection}. 

We obtain the cuspidal functional equation as a byproduct of our results on the sign of periods in Section \ref{sec abstract setup} (which gives a new proof of the functional equation when~$R=\mathbb{C}$), and we prove in Section \ref{sec GJFE} the multiplicativity relation, which also follows from Kondo's formula when $R=\CC$, following Godement and Jacquet.  From this, see for example Corollary \ref{coro mult GJ}, we obtain generalizations of Macdonald's results to modular representations.  Section \ref{sec GJFE} also provides an explicit formula for gamma factors of irreducible representations over algebraically closed fields of characteristic different from~$p$. 
%
%In this paper we provide a different approach to Macdonald's results; we obtain the cuspidal functional equation as a byproduct of our results on the sign of periods in Section \ref{sec abstract setup}, and we prove in Section \ref{sec GJFE} the multiplicativity relation, which also follows from Kondo's formula when $R=\CC$, following Godement and Jacquet. 
%From this, in particular from Corollary \ref{coro mult GJ}, one recovers and generalizes Macdonald's results to representations over more general rings. 
%
%
%Let~$\k$ be a finite field of characteristic~$p$ and cardinality~$q$, and~$R$ be an algebraically close field of characteristic~$\ell\neq p$. Let~$\psi:\k\to R^\times$ be a non trivial character, %and~$\pi$ be an irreducible representation of~$G=\GL_n(\k)$. When~$R$ has characteristic zero, Kondo defined in \cite{Kondo} the non abelian Gauss sum~$\gamma(\pi,\psi)$, and computed it in terms of the Green parameter of~$\pi$. More recently, a proof of geometric nature was given in \cite{BK}, initiating much research on the topic of non abelian Gauss sums. %In particular, 

\subsection{Gamma factors of distinguished representations}\label{sec intro dist gamma} Then we move on to study gamma factors of distinguished representations. Our main motivation is to apply the present paper to classify cuspidal representations of non-archimedean $\GL_n$ distinguished by a Galois involution. This is possible thanks to our previous paper \cite{KMS}, which reduces the problem to the depth zero situation. We note that for complex representations finite field Rankin-Selberg gamma factors and distinction problems have been considered before, although not much, see for example \cite{Nienquadratic} for a relative converse theorem, and \cite{Joperiod}. 

One of the main results of this paper sharpens \cite[Lemma 5.5]{Joperiod}, and answers the question raised before this lemma by Jo in ibid.  We prove that if a cuspidal irreducible representation of~$\GL_n(\k)$ over an algebraically closed field of characteristic different from~$p$ admits either a Galois model, a linear model, or a twisted linear model, then its Godement--Jacquet gamma factor\footnote{Perhaps, it would be more appropriate in this setting to call this gamma factor the \emph{Kondo gamma factor}. However, to stay consistent with the~$p$-adic setting and emphasize the importance of the functional equation, we call this the Godement--Jacquet gamma factor.}  is the sign of the action of the normalizer of the corresponding symmetric subgroup on the corresponding period.  Moreover, we compute the Godement--Jacquet gamma factor of such representations in terms of their Green/James parameters. 

While in this paper, we eventually specialize our computations to the case where~$R$ is a field, as this is the case needed for our application in \cite{KMSII}, our motivation to develop gamma factors for representations in the broader setting considered in Sections \ref{GJFEs}-\ref{sec sym pair}, lies in the further study of distinguished representations and functoriality in finite field analogues of the local Langlands correspondence in families (cf.,~\cite{MR3867634,DHKMbanal} in the~$p$-adic setting and~\cite{MR4580514,MR4698317} for finite fields), and a generalization of the relative converse theorem using families.
%
%then can one compute its Godement--Jacquet gamma factor, and how does this gamma factor relate to the sign character associated to the finite distinguishing period?
 
\subsection{Our main results}
We state here a special case of our main result:% for algebraically closed fields and cuspidal representations. 
\begin{theo}\label{introtheorem1}
Let~$\pi$ be an irreducible cuspidal~$R$-representation of~$\GL_n(\k)$ with~$n\geqslant 2$, suppose the characteristic~$p$ of~$\k$ is odd, and ~$R$ is an algebraically closed field of characteristic different to~$p$.\footnote{When defining the gamma factor one chooses a square root of~$|M_n(\k)|$, this choice appears in the Godement--Jacquet functional equation as we use a normalized Fourier transform.  In this Theorem, in both parts,~$|M_n(\k)|$ is a square of an integer, and we choose the square root in~$R$ agreeing with the positive square root in~$\mathbb{Z}$.}
\begin{enumerate}
\item (Galois distinction) Suppose that~$\k/\k_0$ is quadratic, and~$\pi$ is distinguished by~$\GL_n(\k_0)$. 
 Let~$\psi_0:\kk_0\to F^\times$ be a non-trivial character, then \[\gamma(\pi,\psi_0\circ \tr_{\k/k_0})=\omega_{\pi}(\delta),\]
 where~$\delta\in\kk$ is an element with trace~$\tr_{\kk/\kk_0}(\delta)=0$. 
%Then, if~$\psi$ is a non-trivial character of~$\k$ trivial on~$\k_0$, we have \[\gamma(\pi,\psi)=1.\] 
\item\label{mainth2} (Linear, and twisted-linear distinction) Suppose~$n=2m$ and~$H=\GL_m(\k) \times \GL_m(\k)$, or~$H=\GL_m(\ll)$ for~$\ll/\kk$ a quadratic extension.  Then the normalizer~$N_{\GL_n(\k)}(H)$ acts on~$\Hom_{R[\H]}(\pi,R)$ by a character~$\chi_{\pi}$ (cf.,~Section \ref{sec linear models}).   Let~$\psi$ be a non-trivial character of~$\kk$.  Then
\begin{align*}
\gamma(\pi,\psi)=\chi_{\pi}(w),
\end{align*}
where~$w={\rm antidiag}(1_m,1_m)$ if~$H=\GL_m(\k) \times \GL_m(\k)$, and~$w={\rm diag}(1_m, -1_m)$ if~$H=\GL_m(\ll)$, (embedded in~$G$ as in Section \ref{sec linear models}).
%Suppose~$\pi$ admits a linear or  twisted linear model, as in Section \ref{sec linear models}, then~$\pi$ is self-dual.  Moreover, writing~$\pi=\st_r(\rho)$ with~$\rho$ supercuspidal and~$n=rf$ as in Section \ref{}, we have \rob{check cuspidal non-supercuspidal formula }
%\begin{equation*}
%\g(\pi,\psi)=\begin{cases}
 %(-\xi'(\d))^r&\text{ if $f \geqslant 2$};\\
 %\xi'(-1)^m&\text{ if $f=1$}.\end{cases}
%\end{equation*}
\end{enumerate}
\end{theo}
  To prove the theorem, we use a Poisson summation formula which also provides a new proof of the Godement--Jacquet functional equation for cuspidal representations, see Section \ref{sec GJcusp}.  
  
% Adapting the argument of Theorem \ref{introtheorem1}, for irreducible cuspidal representations over algebraically closed fields of characteristic different to~$\k$, in Section \ref{sec GJcusp} we obtain a different proof of the Godement--Jacquet functional equation using Poisson summation.

\subsection{$p$-adic inspiration for our approach} Let~$F/F_0$ be a quadratic extension of non-archimedean local fields,~$D_0$ be the non split quaternionic~$F_0$-algebra,~$H=\GL_n(F_0)$ or~$\GL_{n/2}(D_0)$ (in which case~$n$ is assumed to be even), and put~$G=\GL_n(F)$. Fix moreover~$\xi:F\to \mathbb{C}^\times$ a non-trivial character, trivial on~$F_0$, and~$\pi$ an~$H$-distinguished cuspidal representation of~$G$. Then, when~$n=2$, amongst many results of this flavour, it is proved in \cite{Hak}, that the Godement-Jacquet gamma factor~$\g(1/2,\pi,\xi)$ of \cite{GJ} is equal to one (\cite{Hak} has the assumption that the central character of~$\pi$ is trivial, but it is nowhere used in [ibid.]). In particular \cite[Theorem 5.1]{Hak} proves this triviality result when~$H=D_0^\times$, by using the Poisson summation formula for~$(\M_2(F),D_0,\psi\circ \tr)$. This idea actually goes back to \cite[Proof of Theorem 3.2]{Delboy}, where Deligne reproves the~$n=1$ case due to Fr\"ohlich and Queyrut \cite{FQ}, using Poisson summation.   When~$H=\GL_n(F_0)$, this method was extended by Ok in his thesis \cite[Main Theorem 1]{Ok} to prove that~$\g(1/2,\pi,\psi)=1$ for any~$n$ 

There are two main difficulties in Ok's proof: the first are convergence issues which do not appear in the case~$n=1$ or~$n=2$ with~$H=D_0^\times$, and the second is to explain why one can replace matrix coefficients of~$\pi$ in the Godement-Jacquet functional equation, by relative matrix coefficients. On the other hand, one facilitating aspect of dealing with~$p$-adic fields is that~$H$ is equal to~$\M_n(F)$ up to a set of measure zero. 

To prove Theorem \ref{introtheorem1}, we adapt the same method to our distinction problems, %i.e., the Poisson summation formula,
 %to prove similar results replacing~$F_0$ by a finite field~$\k_0$, 
 replacing~$F_0$ by the finite field~$\k_0$ and generalizing our coefficient field from~$\mathbb{C}$ to any algebraically closed field of characteristic different to that of~$\k_0$. In this situation, the two difficulties in Ok's proof immediately disappear, however a new difficulty appears:~$H$ is not equal to~$\M_n(\k_0)$ up to a set of measure zero. To circumvent this, we specialize to cuspidal representations, and use cuspidality to kill enough contributions in our computation, see Section \ref{sec abstract setup}, and especially Proposition \ref{prop main}  for the details. 

\subsection{Kondo's formula and computing the sign}
For~$R$ an algebraically closed field, the irreducible (Galois/linear/twisted-linear)-distinguished supercuspidal~$R$-representations are either self-dual (linear/twisted-linear) cases, or Galois-self-dual (Galois case).  Using this and Green's and James' parametrizations of cuspidal~$R$-representations in Section \ref{computation}, we compute the gamma factors appearing in Theorem \ref{introtheorem1}.  Using multiplicativity of gamma factors, we reduce to the supercuspidal case, where we give two methods -- one using the sign of periods description, the other directly from Kondo's formula.  

\subsection{Application to non Archimedean general linear groups}\label{Future}
In \cite{KMS}, for $F$ a non Archimedean local field of residual characteristic not $2$ and $\ell$ a prime different from this residual characteristic, we reduced the classification of~$\ell$-modular cuspidal representations of non-archimedean~$\GL_n$ distinguished by a Galois involution to depth zero.  In the follow up paper \cite{KMSII}, we classify these distinguished representations, with the main result of this paper as a crucial input.

\subsection{Acknowledgements}
Robert Kurinczuk was supported by the Engineering and Physical Sciences Research Council (EPSRC) grant EP/V001930/1. Nadir Matringe was supported by the Research Start-up Fund of the Shanghai Institute for Mathematics and Interdisciplinary Sciences (SIMIS).  Vincent S\'echerre was supported by the Institut Universitaire de France.

\section{Weak and Strong Godement--Jacquet Functional Equations}\label{GJFEs}
We recall that $\k$ is a finite field of characteristic $p$. Let~$\Gg$ be a finite separable~$\k$-algebra, i.e.,~$\Gg$ is a finite product of matrix algebras over finite extensions of~$\k$,  and we set~$G=\Gg^\times$. The algebra~$\Gg$ is equipped with its canonical surjective trace
\[\tr_{\Gg/\k}:\Gg\rightarrow \k.\] %We set~$G=\Gg^\times$. 

We denote by $\mu_p$ the group of roots of $p$-roots of unity in $\CC^\times$. Let $R$ be a commutative $\mathbb{Z}[\mu_p]$-algebra. For $\psi_p\in \Hom(\kk,\CC^\times)=\Hom(\kk,\mathbb{Z}[\mu_p]^\times)=\Hom(\kk,\mu_p)$, we set $\Psi_p=\psi_p\circ \tr_{\Gg/\k}$. Such a character $\psi_p$  defines a character $(\psi_p)_R:\k\rightarrow R^\times$, which we simply denote by $\psi_R$. As a convention, we denote by $\Psi_R:\Gg\rightarrow R^\times$ the character defined by the formula  \[\Psi_R=\psi_R\circ \tr_{\Gg/\k}.\] 
We set
\[\Hb(\kk, R^\times)=\{\psi_R, \  \psi_p\in\Hom(\kk,\mu_p)\}\] 
to be the group of \emph{spectral characters} which we use in this paper to do Fourier theory.\\

\noindent \textit{{\bf Assumption 1:} from now on we assume that the natural map~$\mu_p\to R^\times$ is injective, i.e., that $\Hb(\kk,R^\times)$ contains a non trivial element.}\\

Without this assumption, we could still do Fourier theory with respect to the trivial character, but this becomes too exotic for serious applications. 

For $t\in \k$,~$\psi\in\Hom(\kk,R^\times)$, we write $\psi_t:x\in \k\to \psi(tx)\in R^\times$. 

%In particular, we have~$\Hb(\Gg,R^\times)=\{\psi_R= (\psi_p \circ  \tr_{\Gg/\k})\otimes R:\psi_p\in \Hom(\kk,\mu_p)\}}$
\begin{rema}[Duality]\label{duality}
Let~$\psi_p\in\Hom(\kk,\mu_p)$ be non-trivial.  Then the map~$x\mapsto \psi_{p,x}$ induces a bijection between~$\kk$ and~$\Hom(\kk,\mu_p)$. Putting $\psi:=\psi_R$, we deduce that 
\[\Hb(\kk,R^\times)=\{\psi_x: x\in \kk\}.\]
Moreover, the natural map~$\kk\rightarrow \Hb(\kk,R^\times)$ defined by~$x\mapsto \psi_x$ is bijective.
\end{rema}

\begin{rema}
For this paper the examples of most interest to us are cases where $R$ is an integral domain (in fact, we are most interested in the case of algebraically closed fields).   In this case, we have the equality~{$\Hb(\kk,R^\times)=\Hom(\kk,R^\times)$.}
\end{rema}

 When clear from context we will often omit the index $R$, and write $\psi=\psi_R$ and $\Psi=\Psi_R$.

{\begin{rema}[Orthogonality of character values]\label{orthcharvalues}
If~$\psi=(\psi_p)_R\in \Hb(\kk,R^\times)$, with $\psi_p$ non-trivial, then $\sum_{y\in\Gg}\Psi(y)=0$ since $\sum_{y\in\Gg}\Psi_p(y)=0$.
\end{rema}}

\subsection{The Fourier transform}\label{fouriersection}

We denote by~$\Cc(\Gg,\R)$ the space of functions 
\[\Cc(\Gg,\R):=\{f:\Gg\rightarrow R\},\]
which we consider as a left~$R[G]$-module via the right regular action~$g\cdot f(x)=f(xg)$.  For~$\mathcal{X}\leqslant \Gg$ a~$G$-stable subset, we denote by~$\Cc(\mathcal{X},R)$ the submodule of~$\Cc(\Gg,\R)$ consisting of~$R$-valued functions with support in~$\mathcal{X}$.
 For~$f\in \Cc(G,R)$, we define~$f^\vee\in \Cc(G,R)$ by, for all~$g\in G$, 
 \[f^\vee(g)=f(g^{-1}).\] 
For~$m\in \Gg$, we define~$\delta_m\in \Cc(\Gg,\R)$ to be the indicator function of~$\{m\}$. \\

\noindent\textit{{\bf Assumption 2:} from now on we moreover suppose that~$R^\times$ contains a square root $|\Gg|^{1/2}$ of $|\Gg|$, which we fix. This assumption implies that $R$ has characteristic different from $p$.  Note that Assumption 2 implies Assumption 1.  }\\

{We fix a non-trivial character in $\psi\in \Hb(\kk,R^\times)$.} For~$\Phi\in \Cc(\Gg,\R)$, we define its \emph{(normalized) Fourier transform}~$\mathcal{F}_\psi(\Phi)\in \Cc(\Gg,\R)$, with respect to~$\psi$, by
\begin{align}
\label{Fourier}
\mathcal{F}_\psi(\Phi)( a)&=
|\Gg|^{-1/2}
 \sum \limits_{g\in\Gg} \Phi(g) \Psi(g a) =|\Gg|^{-1/2}\sum \limits_{g\in\Gg} \Phi(g) \Psi( ag) ,
\end{align}
the equality because~$\tr_{\Gg/\k}(ga)=\tr_{\Gg/\k}(ag)$.

The Fourier transform enjoys the following key properties:
\begin{itemize}
\item ($\mathcal{F}_\psi$ \emph{transforms convolution into multiplication}): For~$\Phi, \Phi'\in \Cc(\Gg,\R)$, set
\[\Phi\star\Phi'(x)=|\Gg|^{-1/2}\sum_{y\in \Gg}\Phi(y)\Phi'(x-y),\]
then we have
\begin{equation}\label{eq conv four} \mathcal{F}_\psi(\Phi\star \Phi')=\mathcal{F}_\psi(\Phi)\mathcal{F}_\psi(\Phi').\end{equation}
\item (\emph{The Fourier inversion formula}): As~$\sum_{g\in\Gg}\Psi(g)=0$, cf., Remark \ref{orthcharvalues}, we find that
\[\mathcal{F}_{\psi^{-1}}(\mathcal{F}_\psi(\Phi))=\Phi.\]
\item (\emph{Parseval's identity}): Applying Fourier inversion to Equation \eqref{eq conv four} with~$\Phi'$ replaced the function~$\Phi'_{-}(x)=\Phi'(-x)$, then evaluating at~$x=0$, one finds
\[\sum_{x\in G}\Phi(x)\Phi'(x)=\sum_{x\in G}\mathcal{F}_\psi(\Phi)(x)\mathcal{F}_{\psi^{-1}}(\Phi')(x).\] 
%
%\[\sum_{x\in G}\Phi(x)\Phi'(-x)=\sum_{x\in G}\widehat{\Phi'}^\psi(\Phi)(x)\widehat{\Phi}^\psi(\Phi')(x).\] 
\end{itemize}

When~$\psi$ is fixed, we will use the standard notation
\[\widehat{\Phi}:=\mathcal{F}_{\psi}(\Phi).\] 

\subsection{Godement--Jacquet Functional Equations}
Let~$(\pi,V)$ be an~$R$-representation of~$G$.  We denote by~$(\pi^\vee,V^\vee)$ its contragredient~$R$-representation, i.e., the~$R$-representation on~$V^\vee=\Hom_{R}(V,R)$ defined by~$(\pi^\vee(g)v^\vee)(v)=v^{\vee}(\pi(g^{-1})v)$.  
%
%
%For~$(\pi,V)$ a (finite dimensional) representation of~$G$ with coefficients in~$R$, we denote by~$(\pi^\vee,V^\vee)$ its contragredient representation.

 For~$v\in V$ and~$v^\vee\in V^\vee$, we define the \emph{matrix coefficient}~$f_{v,v^\vee}\in \Cc(G,R)$ by the formula 
\[f_{v,v^\vee}(g)=v^\vee(\pi(g)v).\] 

\begin{defi}
For~$\Phi\in \Cc(\Gg,R)$, and~$(v,v^\vee)\in V\times V^\vee$, define the \emph{Godement-Jacquet sum} 
\begin{equation}\label{eq GJ sum}\Zz(\Phi,f_{v,v^\vee})=\sum\limits_{g\in G} \Phi(g)f_{v,v^\vee}(g).\end{equation}
\end{defi}

We say that~$(\pi,V)$ \emph{satisfies Schur's lemma} if the natural map~$R\rightarrow \End_{R[G]}(\pi)$ defined by~$r\mapsto r \,\Id_V$, where~$\Id_V\in\End_{R[G]}(V)$ is the identity on~$V$, is an isomorphism.\footnote{For example, this is the case if~$R$ is a field of characteristic~$\ell\neq p$ which is sufficiently large for~$G$ (containing all~$|G|$-roots of unity of~$\ell$-regular order) and~$\pi$ is irreducible.}  

Suppose that~$(\pi^\vee,V^\vee)$ satisfies Schur's lemma, then following Kondo \cite{Kondo}, we define the \emph{gamma factor of~$\pi$ with respect to~$\psi$} to be the element of~$R$ defined by 
\begin{equation}\label{eq BK gamma} \gamma^K(\pi,\psi)\Id_{V^\vee}=|\Gg|^{-1/2}\sum_{g\in G}\Psi(g)\pi^\vee(g)\in \End_{R[G]}(V^\vee).\end{equation} 
%where

\begin{prop}[{cf.,~\cite[Proposition 1.2]{CurtisShinoda}}]\label{prop WGJFE irr}
Let~$(\pi,V)$ be an $R$-representation of~$G$ such that $\pi^\vee$ satisfies Schur's lemma, then for any~$\Phi\in \Cc(G,R)$ and any~$(v,v^\vee)\in V\times V^\vee$, we have 
\begin{equation}\label{eq WGJFE}
\Zz(\widehat{\Phi},(f_{v,v^\vee})^\vee)=\gamma^K(\pi,\psi)\Zz(\Phi,f_{v,v^\vee}).
\end{equation}
\end{prop}
\begin{proof}
From equation \eqref{eq BK gamma}, we deduce that for any~$(v,v^\vee)\in V\times V^\vee$, we have 
\[|\Gg|^{-1/2}\sum\limits_{g\in G} \Psi(g)(f_{v,v^\vee})^\vee(g)=\gamma^K(\pi,\psi)v^\vee(v),\] which is equivalent to 
\begin{equation}\label{eq WGJFE0} 
    \Zz(\widehat{\delta_{1}},f_{v,v^\vee}^\vee)=\gamma^K(\pi,\psi)\Zz(\delta_{1},f_{v,v^\vee})=\gamma^K(\pi,\psi)f_{v,v^{\vee}}(1).
\end{equation} The extension from~$\delta_1$ to all~$\Phi\in \Cc(G,R)$ follows from replacing~$v$ by~$\Phi*v=\sum\limits_{x\in G}\Phi(x)\pi(x)v$. 
\end{proof}

\begin{defi}\label{df WGJFEpsi}
Let~$(\pi,V)$ be an~$R$-representation of~$G$. We say that~$\pi$ satisfies the {weak Godement-Jacquet functional equation with respect to~$\psi$}, in short~$(\mathbf{WGJFE},{\psi})$, if there exists~$\gamma^{GJ}(\pi,\psi)\in R$ such that
\[\Zz(\widehat{\Phi},(f_{v,v^\vee})^\vee)=\gamma^{GJ}(\pi,\psi)\Zz(\Phi,f_{v,v^\vee}),\]
%
 %Equality \eqref{eq WGJFE} with $\gamma^{GJ}(\pi,\psi)$ in place of $\gamma^K(\pi,\psi)$ 
for all~$\Phi\in \Cc(G,R)$ and all~$(v,v^\vee)\in V\times V^\vee$, where~$\widehat{\Phi}=\mathcal{F}_\psi(\Phi)$.
\end{defi}

Assume that~$\pi$ satisfies~$(\mathbf{WGJFE},{\psi})$, then Equation \eqref{eq WGJFE0} says that for all~$(v,v^\vee)\in V\times V^\vee$:
\[v^{\vee}\left(\left(|\Gg|^{-1/2}\sum_{g\in G}\Psi(g)\pi(g^{-1})-\gamma^{GJ}(\pi,\psi)\Id_V\right)\cdot v\right)=0.\] 

In other words.

\begin{prop}\label{fct WGJFE}
Let~$(\pi,V)$ be an~$R$-representation of~$G$ which satisfies~$(\mathbf{WGJFE},{\psi})$. Then~$|\Gg|^{-1/2}\sum_{g\in G}\Psi(g)\pi^\vee(g)$ is a scalar operator in~$\End_{R[G]}(V^\vee)$. Moreover, if $\pi^\vee$ satisfies Schur's lemma, then~$\gamma^{GJ}(\pi,\psi)=\gamma^K(\pi,\psi)$ .  
\end{prop}

\begin{rema}\label{rem wgjfe iff wgjfe0}
By the proof of Proposition \ref{prop WGJFE irr}, we see that~$\pi$ satisfies~$(\mathbf{WGJFE},{\psi})$ if and only if the simpler equation \eqref{eq WGJFE0} holds for~$\pi$. 
\end{rema}

From now on, whenever~$\pi$ satisfies~$(\mathbf{WGJFE},{\psi})$, we simply write \[\gamma(\pi,\psi):=\gamma^{GJ}(\pi,\psi),\] which is equal to $\gamma^K(\pi,\psi)$ as soon as $\pi^\vee$ satisfies Schur's lemma.

\begin{defi}\label{df WGJFE}
We say that an~$R$-representation satisfies $\mathbf{WGJFE}$ if it satisfies~$(\mathbf{WGJFE},{\psi'})$ for all non trivial characters~$\psi'\in \Hb(\k,R^\times)$.
\end{defi}

We say that an~$R$-representation~$(\pi,V)$ has a \emph{central character under $\k^\times$}, if the restriction of~$\pi$ to $\k^\times \subseteq Z(G)$ (where~$Z(G)$ denotes the centre of $G$) acts by a character~$\omega_{\pi}:\k^\times\rightarrow R^\times$ on V. This is automatic when $\pi$ satisfies Schur's lemma. We recall from Remark \ref{duality} the equality $\Hb(\k,R^\times)-\{\mathbf{1}\}=\{\psi_t, \ t\in \k^\times\}$. 

\begin{prop}\label{fa A}
%Let~$\psi$ be a non trivial character of~$\k$.
Let~$(\pi,V)$ be an $R$-representation of~$G$ with central character under $\k^\times$.
\begin{enumerate}
\item  Then~$\pi$ satisfies $\mathbf{WGJFE}$ if and only if~$\pi$ satisfies~$(\mathbf{WGJFE},{\psi})$ for our fixed choice of non trivial~$\psi$. 
\item\label{fa A2}  If~$\pi$ satisfies~$\mathbf{WGJFE}$, then, for any~$t\in \k^\times$, we have
 \begin{equation}\label{eq twisting psi in gamma} \g(\pi,\psi_t)=\omega_\pi(t)^{-1}\g(\pi,\psi).
\end{equation} 
 \item 
If~$\gamma(\pi,\psi)\in R^\times$, and the pairing~$V\times V^\vee\rightarrow R$ of~$R$-modules is non-degenerate, the converse of \eqref{fa A2} is also true, i.e., if~$\pi$ satisfies $\mathbf{WGJFE}$, then~$\k^\times$ acts by a character on~$V$.\end{enumerate}
\end{prop}
\begin{proof}
By~$\k$-linearity of the trace~$\Psi_t(ga)=\psi(t\tr_{\Gg/\k}(ga))=\psi(\tr_{\Gg/\k}(tga))=\Psi(tga)$, and we find
\[\mathcal{F}_{\psi_t}(\Phi)(a)=
|\Gg|^{-1/2}
\cdot \sum \limits_{g\in\Gg} \Phi(g) \Psi(tga)=|\Gg|^{-1/2}
\cdot \sum \limits_{g\in\Gg} \Phi(t^{-1}g)\Psi(ga)
=\omega_{\pi}(t)^{-1}\mathcal{F}_{\psi}(\Phi) .\]
The first two statements follow. The third statement follows from Equation \eqref{eq WGJFE0}, as we can choose~$(v,v^{\vee})\in V\times V^{\vee}$ such that~$f_{v,v^{\vee}}(1)=v^{\vee}(v)\in R^\times$ by non-degeneracy of the pairing.
\end{proof}

\begin{rema}
From Propositions \ref{prop WGJFE irr} and \ref{fa A}, we deduce that all absolutely irreducible~$R$-representations of~$G$ over a field of characteristic~$\ell\neq p$ satisfy $\mathbf{WGJFE}$. 
\end{rema}

Now we consider the strong (i.e., the usual) Godement-Jacquet functional equation, where we demand that the functional equation is satisfied for all~$\Phi\in\Cc(\Gg,R)$ rather than just those with support in~$G$:

\begin{defi}\label{df GJFE}
\begin{itemize}
\item Let~$(\pi,V)$ be an~$R$-representation of~$G$, we say that~$\pi$ satisfies the \emph{strong Godement-Jacquet functional equation with respect to~$\psi$}, in short \emph{$(\mathbf{GJFE},{\psi})$}, if there exists~$\gamma(\pi,\psi)\in R$ such that
\begin{equation}\label{eq GJFE0} 
\Zz(\widehat{\Phi},(f_{v,v^\vee})^\vee)=\gamma(\pi,\psi)\Zz(\Phi,f_{v,v^\vee}),\end{equation}
for all~$\Phi\in \Cc(\Gg,R)$ and all~$(v,v^\vee)\in V\times V^\vee$.
\item We say that~$\pi$ satisfies~$\mathbf{GJFE}$ if it satisfies~$(\mathbf{GJFE},{\psi'})$ for all non-trivial~$\psi'\in \Hb(\k, R^\times)$.\end{itemize}
%(In other words, Equality \eqref{eq WGJFE} holds for~$\pi$, with~$\Phi$ varying in~$ \Cc(\Gg,R)$ instead of~$R[G]$. )
\end{defi}

\begin{lemm}\label{lemmgammagammatilde}
Let~$(\pi,V)$ be an~$R$-representation of~$G$, and suppose that~$\pi$ satisfies~$(\mathbf{GJFE},{\psi})$ and that~$\pi^{\vee}$ satisfies~$(\mathbf{GJFE},{\psi}^{-1})$.  Then~$\gamma(\pi,\psi)\in R^\times$ and 
\begin{equation}\label{eq dist gr} \g(\pi,\psi)\gamma(\pi^\vee,\psi^{-1})=1.\end{equation}  
\end{lemm}
\begin{proof}
This follows from Fourier Inversion, by applying the Godement-Jacquet functional equation twice.
\end{proof}
%
%It follows from applying the functional equation twice, that if~$\pi$ satisfies~$(GJFE)_{\psi}$ and~$\pi^\vee$ satisfies~$(GJFE)_{\psi^{-1}}$, then~$\gamma(\pi,\psi)\in R^\times$ and 
%\begin{equation}\label{eq dist gr} \g(\pi,\psi)\gamma(\pi^\vee,\psi^{-1})=1.\end{equation} 

Mutatis mutandis the proof of \ref{fa A}, we find:

\begin{prop}
Let~$(\pi,V)$ be an~$R$-representation of~$G$ with central character under $\k^\times$.  
%
%, and suppose that~$\k^\times$ acts on~$V$ by a character~$\omega_\pi:\k^\times\rightarrow R^\times$. 
Then~$\pi$ satisfies~$\mathbf{GJFE}$ if and only if~$\pi$ satisfies~$(\mathbf{GJFE},{\psi})$.
\end{prop}

\begin{rema}
We warn the reader that irreducible representations over algebraically closed fields do not automatically satisfy $\mathbf{GJFE}$, and let them 
discover why in the following sections. However when $G$ is non-commutative, cuspidality ensures $\mathbf{GJFE}$ as we prove in Proposition \ref{prop cuspidal GJFE}. 
\end{rema}

\section{The sign of periods and the cuspidal functional equation}\label{sec abstract setup}

We recall that we have fixed a non-trivial character~$\psi\in \Hb(\k,R^\times)$ and~$\Psi=\psi\circ\tr_{\Gg/\k}$.

\subsection{The sign of periods via Poisson summation}

We let~$\Hh$ be a~$\k$-vector subspace of~$\Gg$. We set~$\Hh^\bot=\{ b \in\Hh\ |\ \Psi(ab)=1 \text{ for all } a\in\Hh \}$, it is a~$\k$-vector space which identifies with the~$\k$-dual of~$\Hh\backslash \Gg$, and in particular~$\dim_{\kk}(\Hh)+\dim_{\kk}(\Hh^\bot)=\dim_{\kk}(\Gg)$, hence~$|\Hh|\cdot|\Hh^\bot|=|\Gg|$.

\begin{lemm}
For any~$x\in G$, we have the Poisson summation formula
\begin{equation}
\label{Poisson} 
\sum\limits_{a\in \Hh} \Phi(ax) =
|\Hh| \cdot |\Gg|^{-1/2} \cdot
\sum\limits_{b\in \Hh^\bot} \widehat{\Phi}(x^{-1}b).
\end{equation}
\end{lemm}

\begin{proof}
First note that it suffices to prove it when~$x=1$ for any~$\Phi$,
since the Fourier transform of the function~$g \mapsto \Phi(gx)$ 
is the function~$\a \mapsto \widehat{\Phi}(x^{-1}\a)$. The Poisson formula now follows from applying Parseval identity to~$\Phi$ and~$\bf{1}_{ \Hh }$ the characteristic function of~$\Hh$, the Fourier transform of which is equal to~$|\Hh| \cdot |\Gg|^{-1/2} \cdot \bf{1}_{ \Hh^\bot }$. 
\end{proof}

\subsubsection{}

Let~$(\pi,V)$ be an~$R$-representation of~$G$. {Let~$H$ be a subgroup of $G$}, and assume the existence of a non-zero~$H$-invariant linear form
\[\La\in V^\vee-\{0\} = \Hom_{R}(\pi,R)-\{0\}.\]

The matrix coefficient~$f_{v,\La}$ is then a left~$H$-invariant function in~$\Cc(G,R)$. For~$\Phi\in \Cc(\Gg,R)$,
we consider the Godement--Jacquet sum
\begin{equation}
\label{Emmanuel}
\Zz(\Phi,f_{v,\La}) = \sum\limits_{g\in G} \Phi(g) f_{v,\La}(g) = 
\sum\limits_{x\in H \backslash G} \left( \sum\limits_{h\in H}\Phi(hx) \right) f_{v,\La}(x).
\end{equation}

\subsubsection{}\label{cuspkills}

Now suppose that~$H$ is contained in a subspace~$\Hh$ of~$\Gg$
which is stable under right multipli\-cation by~$H$. 
We look for a sufficient condition to have
\begin{equation}
\label{linearisation}
\Zz(\Phi,f_{v,\La}) = 
\sum\limits_{x\in H \backslash G} \left( \sum\limits_{a\in \Hh}\Phi(ax) \right) f_{v,\La}(x).
\end{equation}
Given any~$H$-orbit~$\oo$ in~$\Hh$ (where~$H$ acts on the right), 
consider the sum
\begin{equation*}
J(\Phi,f_{v,\La},\oo) = \sum\limits_{x\in H \backslash G} \left(
\sum\limits_{a\in \oo}\Phi(ax) \right) f_{v,\La}(x). 
\end{equation*}
Note that~$H$ is a single~$H$-orbit and that~$J(\Phi,f_{v,\La},H)=\Zz(\Phi,f_{v,\La})$.
On the one hand, we have
\begin{equation*}
\sum\limits_{x\in H \backslash G} \left( \sum\limits_{a\in \Hh}\Phi(ax)
\right) f_{v,\La}(x) = \sum\limits_{\oo} J(\Phi,f_{v,\La},\oo).
\end{equation*}
We observe that for fixed an~$a\in\oo$, 
we have
\begin{eqnarray*}
J(\Phi,f_{v,\La},\oo) &=& \sum\limits_{x\in H \backslash G} \left(
\sum\limits_{h\in {}_aH \backslash H}\Phi(ahx) \right) f_{v,\La}(x) \\ &=& 
\sum\limits_{y\in {}_aH \backslash G} \Phi(ay) f_{v,\La}(y) \phantom{\left(
\sum\limits_{h\in {}_aH \backslash H}\Phi \right) } \\ &=& 
\sum\limits_{r\in {}_aG \backslash G} 
\left( \sum\limits_{s\in {}_aH \backslash {}_aG} f_{v,\La}(sr) \right) \Phi(ar)
\end{eqnarray*}
where~${}_aG$ is the subgroup~$\{g\in G\ |\ ag=a\}$ 
and~${}_aH=H \cap {}_aG$.

\subsubsection{}

For $J$ a subgroup of $G$ and $(\pi,V)$ an $R$-representation of $G$, we denote by $V^J$ the space of $J$-invariant vectors in $V$, and by $V_J$ the~$J$-coinvariant space. We recall the definition of cuspidality:

\begin{defi} Let $(\pi,V)$ be an $R$-representation of $G$, we say that $\pi$ is cuspidal if, whenever $U$ is the unipotent radical of a proper parabolic subgroup of $G$, we have $V_U=\{0\}$.
\end{defi}

\begin{rema}\label{rem J vs dual}
Let~$(\pi,V)$ be an~$R$-representation of~$G$.  Since we are assuming~$p$ is invertible in~$R$, for~$U$ the unipotent radical of a parabolic subgroup of~$G$ we have~$V_U\simeq V^U$, and hence~$\pi$ is cuspidal if and only if~$V^U$ is trivial for all unipotent radicals of all proper parabolic subgroups of~$G$.  Because~$(V^\vee)^U\simeq (V_U)^\vee$, this implies that if~$\pi$ is cuspidal then so is~$\pi^{\vee}$. On the other hand, note that $\pi^{\vee}$ could be cuspidal even if $\pi$ is not. For example, assume that~$R=\Zl$ for $\ell$ a prime number different from $p$ and that $G$ is not commutative, let $\pi_1$ be a non-cuspidal $\Fl$-representation of $G$ seen as a $\Zl$-representation, and let $\pi_2$ be a cuspidal $\Zl$-representation. Then $\pi=\pi_1\oplus \pi_2$ is not cuspidal, whereas its dual $\pi^\vee\simeq \pi_2^\vee$ is cuspidal.
%
%Note that $(V^\vee)^U\simeq (V_U)^\vee$, but since we are assuming that $p$ is invertible in $R$, we have $V_U\simeq V^U$. From this we deduce that if $\pi$ is cuspidal, so is $\pi^\vee$.  
\end{rema}
%\\

Then our conditions which guarantee \eqref{linearisation} are: 
\begin{itemize}
\item~$\pi^\vee$ is cuspidal (which is enforced if~$\pi$ is cuspidal by Remark \ref{rem J vs dual}), and,
\item for any~$a\notin H$, there exists a proper parabolic subgroup of~$G$
whose~uni\-po\-tent radical is contained in~${}_aG$  {(in particular $G$ is not commutative).} 
\end{itemize}
Indeed, if they are satisfied, the~${}_aG$-invariant linear form
\begin{equation*}
v \mapsto \sum\limits_{s\in {}_aH \backslash {}_aG} f_{v,\La}(sr)
\end{equation*}
on~$V$ must be zero by cuspidality of~$\pi^\vee$. 
We thus get~$J(\Phi,f_{v,\La},\oo)=0$ for any orbit
$\oo\neq H$, and Equation \eqref{linearisation} follows as expected. 

\subsubsection{}

From now on, we moreover assume that~$\Hom_{R[H]}(\pi,R)$ is free of rank one over~$R$. 
The~normalizer~$N$ of~$H$ in~$G$ thus acts on~this space
through a character~$\chi_\pi : N \to \R^\times$, more precisely 
\[\La \circ \pi(n)=\chi_\pi(n)\La \] for any~$n\in N$,~$\La \in \Hom_{R[H]}(\pi,\R)$. 
Furthermore, we suppose that~$\Hh^\bot=w\Hh$ for some~$w\in N$. 
We get~$|\Hh^\bot|=|\Hh|$,
and take ~$|\Gg|^{1/2}:=|\Hh|$ in~$R$. By applying the Poisson summation formula \eqref{Poisson}
applied to \eqref{linearisation} we find 
\begin{eqnarray*}
\Zz(\Phi,f_{v,\La}) &=& 
\sum\limits_{x\in H \backslash G} \left( \sum\limits_{\xi\in \Hh^\bot}
  \widehat{\Phi}(x^{-1}\xi) \right) f_{v,\La}(x)\\
  &=& \sum\limits_{x\in H \backslash G} \left( \sum\limits_{a\in \Hh}
                \widehat{\Phi}(x^{-1}wa) \right) f_{v,\La}(x) \\
 & = & 
  \sum\limits_{x\in H \backslash G} \left( \sum\limits_{a\in \Hh}
    \widehat{\Phi}(x^{-1}a) \right) f_{v,\La}(w^{-1}x) \\ 
 & = & 
\chi_{\pi}(w)^{-1} \cdot \sum\limits_{x\in H \backslash G} \left( \sum\limits_{a\in \Hh}
    \widehat{\Phi}(x^{-1}a) \right) f_{v,\La}(x). 
\end{eqnarray*}
If we moreover assume that~$\Hh$ is stable under left multiplication by~$H$ and,
for any~$a\notin H$,
there is a pro\-per parabolic subgroup of~$G$
whose~uni\-po\-tent radical is contained in
$G_a = \{g\in G\ |\ ga=a\}$, following Section \ref{cuspkills}, 
we may apply an analogue of \eqref{linearisation} which gives us
\begin{equation*}
\Zz(\Phi,f_{v,\La})  =  
\chi_{\pi}(w)^{-1} \cdot \sum\limits_{x\in H \backslash G} \left( 
  \sum\limits_{h\in \H}
  \widehat{\Phi}(x^{-1}h) \right) f_{v,\La}(x)
% = \chi_{\pi}(w)^{-1} \cdot \Zz(\Phi^*,v)
= \chi_{\pi}(w)^{-1} \cdot \sum\limits_{g\in G} \widehat{\Phi}(g) f_{v,\La}(g^{-1}).
\end{equation*} Let us summarize.

\begin{prop}\label{prop main}
Let~$\pi$ be an ~$R$-representation of~$G$,
and let~$H$ be a subgroup of~$G$ with the following properties:
\begin{enumerate}[label=(A\arabic*)]
\item \label{A1}  the smooth dual $\pi^\vee$ is cuspidal,
\item \label{A2}
the space~$\Hom_{R[H]}(\pi,\R)$ is free of rank one over~$R$,
\item \label{A3}
the subgroup~$H$ is contained in a subspace~$\Hh$ of~$\Gg$ which is 
stable by left and right multipli\-cation by~$H$, 
\item \label{A4}
for all~$a\in\Hh$ such that~$a\notin H$, 
the subgroup~$\{g\in G\ |\ ga=a\}$ contains the unipotent radical of some proper parabolic subgroup of~$G$, and so does the subgroup~$\{g\in G\ |\ ag=a\}$.
\item \label{A5}
there is an element~$w\in G$ in the normalizer~$N$ of~$H$ such that,
for all~$a\in\Gg$, 
one has~$\tr(a\Hh) \subseteq \Ker(\Psi)$
if and only if~$a\in w\Hh$.  In particular,~$|\Gg|=|\Hh|^2.$
\end{enumerate}
Then, if~$\chi_{\pi}:N\to R^\times$ is the character defined by~$\La\circ \pi =\chi_\pi(w)\cdot \La~$, and we choose ~$|\Gg|^{1/2}:=|\Hh|$ in~$R$, one has 
\begin{equation*}
 \Zz(\widehat{\Phi},f_{v,\La}^\vee)=\chi_{\pi}(w) \Zz(\Phi,f_{v,\La}) 
\end{equation*}
for all~$\La \in \Hom_{R[H]}(\pi,R)$,~$v\in V$, and~$\Phi\in \Cc(\Gg,R)$. In particular, if~$\pi$ satisfies $(\mathbf{WGJFE},{\psi})$, then \[\g(\pi,\psi)=\chi_\pi(w).\] 
\end{prop}

\begin{rema}
 In the above situation, the gamma value~$\g(\pi,\psi)$ does not depend on the choice of~$\psi$. If~$\pi$ has central character and~$\gamma(\pi,\psi)\in R^\times$, then by Proposition \ref{fa A}, this is equivalent to the statement that~$\omega_{\pi}\mid_{\k^\times}$ is trivial.
\end{rema}

\subsection{First application: the cuspidal Godement-Jacquet functional equation}\label{sec GJcusp}
Let $n\geq 2$. We now assume that $G=\GL_n(\k)\times\GL_n(\k)$ and $H=\GL_n(\k)$
diagonally~em\-bed\-ded in $G$. Let $\Pi$ be an $R$-representation of $G$ of the form $\Pi=\pi^{}\otimes\pi^\vee$, 
for some representation $\pi$ of $\GL_n(\k)$ such that $\pi^\vee$ is cuspidal. We moreover assume that $\pi^\vee$ satisfies Schur's lemma.

\ref{A1} for $\Pi$ is satisfied according to Remark \ref{rem J vs dual}, by the assumption  that $\pi^\vee$ is cuspidal.

\ref{A2} is satisfied since $\Hom_{R[H]}(V_\pi\otimes V_\pi^\vee,R)\simeq \Hom_{R[\GL_n(\k)]}(V_\pi^\vee, V_\pi^\vee)$ and $\pi^\vee$ satisfies Schur's lemma.

\ref{A3} is satisfied with $\Hh=\Mat_n(\k)$ diagonally embedded 
in $\Gg=\Mat_n(\k)\times\Mat_n(\k)$.

\ref{A4} is satisfied:
given $a\in\Mat_n(\k)$ and $g\in\GL_n(\k)$, 
one has $ga=a$ if and only if $g$ is of the form $1+u$ where $u$
is any element~of $\Mat_n(\k)$ whose kernel contains the image of $a$,
and we have a similar result for ${}_aG$.

\ref{A5} is satisfied with $w=(1,-1)$.

\noindent

We have a canonical $G$-equivariant map from $\iota:V_\pi^\vee\otimes V_\pi \to V_\Pi^\vee$ defined by \[\iota(x^\vee\otimes x)(v\otimes v^\vee)=\langle v, x^\vee\rangle \langle x, v^\vee\rangle.\] Now choose $v\otimes v^\vee$ and $\iota(x^\vee\otimes x)$
in $V_\Pi$ and $V_\Pi^\vee$, respectively, and $\Phi\otimes \Phi'\in \Cc(\Gg,R)$.
Note that $\Pi(w)=\omega_{\pi}(-1)\Id_{V_\pi\otimes V_\pi^\vee}$. Applying Proposition \ref{prop main} with $\La$ the canonical pairing from $V_\pi\otimes V_\pi^\vee$ to $R$, we get
\begin{equation*}
\Zz(\widehat{\Phi},f_{v,x^\vee}^\vee)\Zz(\widehat{\Phi'},f_{x,v^\vee}^\vee)= \omega_{\pi}(-1)\Zz(\Phi,f_{v,x^\vee})\Zz(\Phi',f_{x,v^\vee}).
\end{equation*} 

Taking~$\Phi'$,~$x$ and~$v^\vee$ such that~$\Zz(\widehat{\Phi'},f_{x,v^\vee}^\vee)=1$, we obtain the cuspidal Godement-Jacquet functional equation. Let us summarize:

\begin{prop}\label{prop cuspidal GJFE}
Let $\pi$ be an $R$-representation of $G$ such that $\pi^\vee$ is cuspidal and satisfies Schur's lemma. Then $\pi$ satisfies $\mathbf{GJFE}$.
\end{prop}

\section{Multiplicativity and the general functional equation}\label{sec GJFE} 
We continue with~$\Gg=\Gg_n=\Mat_n(\k)$ and~$G=G_n=\Gg^\times_n=\GL_n(\k)$. Following \cite{GJ}, we prove the multiplicativity relation of Godement-Jacquet gamma factors for~$G_n$, and deduce from it the general Godement-Jacquet functional equation. The following lemma has an obvious proof.

\begin{lemm}\label{matrixcoeffsqs}
Let~$(\pi,V)$ be an~$R$-representation of~$G_n$, and~$(\pi_2,V_2)\leqslant (\pi_1,V_1)\leqslant (\pi,V)$~$R$-subrepresentations of~$G_n$, with subquotient~$(\pi',V')=(\pi_1,V_1/V_2)$.  Assume moreover that the restriction map from $\Hom_R(V,R)$ to $\Hom_R(V_1,R)$ is surjective.  Then given~$\overline{v}\in V',$~$\overline{v}^{\vee}\in V'^{\vee}$ there exist~$v\in V$,~$v^{\vee}\in V^{\vee}$ such that~$f_{v,v^{\vee}}=f_{\overline{v},\overline{v}^{\vee}}$.
\end{lemm}

In particular, we infer the following. 

\begin{lemm}\label{sqsgamma}
Let~$(\pi,V)$ be an~$R$-representation of~$G_n$ satisfying~$(\mathbf{GJFE},\psi)$ (resp.~$(\mathbf{WGJFE},\psi)$).  Let~$(\pi',V')=(\pi_1,V_1/V_2)$ be a subquotient of~$(\pi,V)$, and assume that the restriction map from $\Hom_R(V,R)$ to $\Hom_R(V_1,R)$ is surjective. Then~$(\pi',V')$ satisfies~$(\mathbf{GJFE},\psi)$ (resp.~$(\mathbf{WGJFE},\psi)$), and moreover
\[\gamma(\pi',\psi)=\gamma(\pi,\psi).\]
\end{lemm}
  
By a simple computation, simplifying the proof of \cite[Lemma 3.4.0]{GJ} to the finite field setting, we have:

\begin{lemm}
Let~$n=n_1+n_2$ with~$n_i\geq 1$. If~$\Phi\in \Cc(\Gg_n,R)$, then the map 
\begin{equation}
\label{rem WGJFE step}
\Pi_{\Phi}(x_1,x_2)=\sum_{x\in \Mat_{n_1,n_2}(\kk)} \Phi \begin{pmatrix} x_1 & x \\ 0 & x_2\end{pmatrix} \end{equation} belongs to~$\Cc(\Gg_{n_1}\times \Gg_{n_2},R)\simeq \Cc(\Gg_{n_1},R)\otimes \Cc(\Gg_{n_2},R)$.   Moreover, we have
\begin{equation}\label{eq GJ habile} \widehat{\Pi_{\Phi}}=\Pi_{\widehat{\Phi}},\end{equation}
where both Fourier transforms are defined with respect to~$\psi$.
\end{lemm}
 
Let~$P_{(n_1,n_2)}$ denote the standard (block upper triangular) parabolic subgroup of~$G_n$ with Levi factor~$G_{n_1}\times G_{n_2}$.  Let~$\pi_i$ be~$R$-representations of~$G_{n_i}$, for~$i=1,2$, then we use the product notation~$\times$ for parabolic induction:
\[\pi_1\times \pi_2:=\mathrm{Ind}_{P_{(n_1,n_2)}}^{G_n}(\pi_1\otimes\pi_2).\]
%the product notation~$\times$ for parabolic induction, 
We now obtain a generalization of \cite[Theorem 3.4]{GJ}, with the same proof, which we recall. 

\begin{theo}\label{thm mult GJ}
Let~$\pi_i$ be~$R$-representations of~$G_{n_i}$ for~$i=1,2$, and~$(\pi,W_1/W_2)$ be a nonzero subquotient of~$(\pi_1\times \pi_2,V_{\pi_1\times \pi_2})$. Assume that the restriction map from $\Hom_R(V_{\pi_1\times \pi_2},R)$ to $\Hom_R(W_1,R)$ is surjective, and that both~$\pi_1$ and~$\pi_2$ satisfy~$(\mathbf{GJFE},\psi)$ (resp.~$(\mathbf{WGJFE},\psi)$). Then~$\pi$ satisfies~$(\mathbf{GJFE},\psi)$ (resp.~$(\mathbf{WGJFE},\psi)$), and moreover we have the multiplicativity relation 
\[\gamma(\pi,\psi)= \gamma(\pi_1,\psi)\gamma(\pi_2,\psi).\]
\end{theo}
\begin{proof}
We start with~$(\mathbf{GJFE},\psi)$.  By Lemma \ref{sqsgamma}, it is sufficient to prove the theorem for~$\pi=\pi_1\times \pi_2$.  
%
%Since, by Lemma \ref{matrixcoeffsqs}, matrix coefficients of~$\pi$ can be realized as matrix coefficients of~$\pi_1\times \pi_2$, we may assume that~$\pi=\pi_1\times \pi_2$. 
Let~$P_{(n_1,n_2)}=M_{(n_1,n_2)}U_{(n_1,n_2)}$ be the standard Levi decomposition of the standard parabolic subgroup~$P_{(n_1,n_2)}$, i.e.,
\[M_{(n_1,n_2)}=\left\{\diag(g_1,g_2),\ g_i\in G_{n_i}\right\},\quad\text{and}~U_{(n_1,n_2)}=\left\{\begin{pmatrix} I_{n_1} & x \\ 0 & I_{n_2} \end{pmatrix}, \ x\in \Mat_{n_1,n_2}(\kk)\right\}.\] 
For 
$\alpha\in \pi_1\times \pi_2$, and~$\alpha^\vee\in (\pi_1\times \pi_2)^\vee \simeq \pi_1^\vee\times \pi_2^\vee$, the corresponding matrix coefficient of~$\pi$ is given by \[f_{\alpha,\alpha^\vee}(g)=\sum\limits_{r\in \mathcal{R}} \langle \alpha(rg),\ \alpha^\vee(r) \rangle,\] where~$\mathcal{R}$ is a set of representatives of~$P\backslash G$ and~$\langle \ ,\  \rangle$ stands for the duality bracket between~$\pi_1\otimes \pi_2$ and 
$\pi_1^\vee \otimes \pi_2^\vee$. 

Now take~$\Phi\in \Cc(G_n,R)$, and write~$(x',x)\cdot \Phi(g)=\Phi(x^{-1}gx')$ for~$x,g,x'\in G$. Then 
\begin{align*}
\Zz(\Phi,f_{\alpha,\alpha^\vee})&=\sum\limits_{g\in G} \sum\limits_{r\in \mathcal{R}} \Phi(g)\langle \alpha(rg),\ \alpha^\vee(r) \rangle= 
\sum\limits_{g\in G} \sum\limits_{r\in \mathcal{R}} \Phi(r^{-1}g)\langle \alpha(g),\ \alpha^\vee(r) \rangle\\
&=\sum\limits_{r' \in \mathcal{R}} \sum\limits_{y\in P} \sum\limits_{r\in \mathcal{R}} \Phi(r^{-1}yr')\langle \alpha(yr'),\ \alpha^\vee(r) \rangle\\
&= \sum\limits_{r' \in \mathcal{R}} \sum\limits_{y\in P} \sum\limits_{r\in \mathcal{R}} ((r',r)\cdot\Phi)(y) \langle \pi_1\otimes \pi_2(y)\alpha(r'),\ \alpha^\vee(r) \rangle\\
&= \sum\limits_{(r,r') \in \mathcal{R}^2}  \sum_{\substack{u\in U_{(n_1,n_2)}\\(g_1,g_2)\in G_{n_1}\times G_{n_2}}} ((r',r)\cdot\Phi)(u\diag(g_1,g_2)) %\sum\limits_{(g_1,g_2)\in G_1\times G_2} 
\langle \pi_1(g_1)\otimes \pi_2(g_2)\alpha(r'),\ \alpha^\vee(r) \rangle \\
&=\sum\limits_{(g_1,g_2)\in G_{n_1}\times G_{n_2}} \sum\limits_{(r,r') \in \mathcal{R}^2} \Pi_{(r',r)\cdot\Phi}(\diag(g_1,g_2))  \langle \pi_1(g_1)\otimes \pi_2(g_2)\alpha(r'),\ \alpha^\vee(r) \rangle
 \\&=\sum\limits_{(g_1,g_2)\in G_{n_1}\times G_{n_2}} \sum\limits_{(r,r') \in \mathcal{R}^2}  \Pi_{(r',r)\cdot\Phi}(\diag(g_1,g_2)) J(g_1,g_2,r,r') ,\end{align*}
where~$J(g_1,g_2,r,r') =\langle \pi_1(g_1)\otimes \pi_2(g_2)\alpha(r'),\ \alpha^\vee(r) \rangle$.  Given,~$(r,r')\in\mathcal{R}^2$, there exist positive integers~$v_{r,r'},j_{r,r'}$ so that we can write
\[J(g_1,g_2,r,r')=\sum\limits_{u=1}^{v_{r,r'}} f_{1,u}^{r,r'}(g_1)f_{2,u}^{r,r'}(g_2),\] 
where each~$f_{i,u}^{r,r'}$ is a matrix coefficient of~$\pi_i$, and write  
\[ \Pi_{(r',r)\cdot\Phi}(\diag(g_1,g_2))=\sum\limits_{i=1}^{j_{r,r'}}\phi_{1,i}^{r,r'}(g_1)\phi_{2,i}^{r,r'}(g_2),\]
where~$\phi_{k,i}^{r,r'}\in\Cc(G_{n_k},R)$.  We finally obtain 
 \[\Zz(\Phi,f_{\alpha,\alpha^\vee})=\sum\limits_{(r,r') \in \mathcal{R}^2} \sum\limits_{u=1}^{v_{r,r'}}\sum\limits_{i=1}^{j_{r,r'}} \Zz(\phi_{1,i}^{r,r'},f_{1,u}^{r,r'}) \Zz(\phi_{2,i}^{r,r'},f_{2,u}^{r,r'}).\]
Similarly, observing that $ \widehat{(r',r)\cdot \Phi}={(r,r')\cdot\widehat{ \Phi}}$ we find~$\widehat{\Pi_{(r',r)\cdot\Phi}}=\Pi_{(r,r')\cdot \widehat{\Phi}}$ by \eqref{eq GJ habile}, and we obtain
\[\Zz(\widehat{\Phi},f_{\alpha,\alpha^\vee}^\vee)=\sum\limits_{(r,r') \in \mathcal{R}^2} \sum\limits_{u=1}^{v_{r,r'}}\sum\limits_{i=1}^{j_{r,r'}} \Zz(\widehat{\phi_{1,i}^{r,r'}},(f_{1,u}^{r,r'})^\vee) \Zz(\widehat{\phi_{2,i}^{r,r'}},(f_{2,u}^{r,r'})^\vee).\] The statement follows. Now, in view of Equation \eqref{rem WGJFE step}, the statement for~$(\mathbf{WGJFE},\psi)$ also follows from the proof above, by observing that if~$\Phi$ is supported on~$G_n$, then~$\Pi_{(r,r')\cdot\Phi}$ is supported in~$\G_{n_1}\times \G_{n_2}$, so that we can take the maps~$\phi_{k,i}^{r,r'}$ in~$\Cc(G_{n_k},R)$, for~$k=1,2$.
\end{proof}

As a first corollary, we obtain.

\begin{coro}\label{coro mult GJ}
Let~$\pi_i$ be a representation of~$G_{n_i}$ for~$i=1,\dots,k$,~$k\geq 1$, and~$(\pi,W_1/W_2)$ be a nonzero subquotient of~$(\pi_1\times\cdots \times \pi_k,V_{\pi_1\times\cdots \times \pi_k})$. Assume that the restriction map from $\Hom_R(V_{\pi_1\times\cdots \times \pi_k},R)$ to $\Hom_R(W_1,R)$ is surjective, and moreover that each~$\pi_i$ satisfies the~$\mathbf{GJFE}$ (resp.~$\mathbf{WGJFE}$). Then~$\pi$ satisfies~$\mathbf{GJFE}$ (resp.~$\mathbf{WGJFE}$), and moreover one has the multiplicativity relation 
\[\gamma(\pi,\psi)= \prod_{i=1}^k\g(\pi_i,\psi).\]
\end{coro}
\begin{proof}
Again, by Lemma \ref{sqsgamma}, we may assume that~$\pi=\pi_1\times\dots \times \pi_k$. The result follows by induction from Theorem \ref{thm mult GJ}. 
\end{proof}

We now recall the following result from \cite[Theorem 3.1]{Roddity} for complex representations, for which we provide a different proof.

\begin{lemm}\label{lm GJ GL1}
Let~$\chi:\kk^\times\to R^\times$ be a character. Then for any~$\Phi\in \Cc(\kk,R)$:
\[\Zz(\widehat{\Phi},\chi^{-1})=\gamma(\chi,\psi) \Zz(\Phi,\chi)+q^{-1/2}\Phi(0)\sum\limits_{x\in \kk^\times}\chi^{-1}(x).\]
In particular,~$\chi$ satisfies~$\mathbf{WGJFE}$, and further satisfies~$\mathbf{GJFE}$ in the following cases:
\begin{itemize}
\item $R$ is an integral domain and~$\chi$ is non-trivial;
\item  $R$ is a field of characteristic~$\ell\mid q-1$.
\end{itemize}
%~$\chi$ satisfies (GJFE) whenever~$\chi$ is non trivial, or when~$\ell$ divides~$q-1$, and it always satisfies (WGJFE).
\end{lemm}
\begin{proof}
For~$\mathbf{WGJFE}$ we refer to Proposition \ref{prop WGJFE irr}, as~$\chi$ satisfies Schur's lemma. Now we prove the formula of the statement,
\begin{align*}
\Zz(\widehat{\Phi},\chi^{-1})&=q^{-1/2}\sum\limits_{x\in \kk^\times}\sum\limits_{y\in \kk} \Phi(y)\psi(yx)\chi^{-1}(x)
\\&=q^{-1/2}\Phi(0)\sum\limits_{x\in \kk^\times}\chi^{-1}(x)+q^{-1/2}\sum\limits_{x\in \kk^\times}\sum\limits_{y\in \kk^\times} \Phi(y)\psi(yx)\chi^{-1}(x).\end{align*} But 
\begin{align*}
q^{-1/2}\sum\limits_{x\in \kk^\times}\sum\limits_{y\in \kk^\times} \Phi(y)\psi(yx)\chi^{-1}(x)&=\left(q^{-1/2}\sum\limits_{x\in \kk^\times}\psi(x)\chi^{-1}(x)\right)\left(\sum\limits_{y\in \kk^\times}\Phi(y)\chi(y) \right)\\
&=  \Zz(\widehat{\delta_1},\chi^{-1})\Zz(\Phi,\chi)=\gamma(\chi,\psi) \Zz(\Phi,\chi).\end{align*}
The final statement follows as they give conditions for the sum~$\sum_{x\in\k^\times}\chi^{-1}(x)$ to vanish.%in the former case the argument is analogous to that of Remark \ref{orthcharvalues}.
\end{proof}

\begin{rema}\label{rem triv gamma}
\begin{enumerate}
\item When~$\chi=\mathbf{1}$ is trivial, then
\[\gamma(\mathbf{1},\psi)=-q^{-1/2}.\]  In particular,~$\gamma(\mathbf{1},\psi)^2=q^{-1}$.
\item When~$q$ is odd,~$\ell\neq 2$, and~$\chi:\kk^\times \rightarrow \{\pm 1\}$ is non-trivial quadratic character valued in~$\{\pm 1\}\subset R^\times$, then we have the classical quadratic Gauss sum, and by applying the functional equation twice we find
\[\gamma(\chi,\psi)^2=\gamma(\chi,\psi)\gamma(\chi^{-1},\psi)=\chi(-1)\gamma(\chi,\psi)\gamma(\chi^{-1},\psi^{-1})=\chi(-1).\]
\end{enumerate}
\end{rema}

 We now recover a generalization to include~$\ell$-modular representations of \cite[2, Proposition]{Macdonald} (see also \cite[Theorem 4.1.1]{Roddity}), as a straightforward consequence of Proposition \ref{prop cuspidal GJFE}, Theorem \ref{thm mult GJ}, Lemma \ref{lm GJ GL1}, and the existence of cuspidal support for irreducible representations.  %Recall, we assume that~$R$ is a~$\mathbb{Z}[\mu_p,|\Gg|^{-1/2}]$-algebra throughout. 

\begin{theo}\label{thm GJFE}
Suppose that~$R$ is an algebraically closed field of characteristic different to~$p$.  Let~$(\pi,V)$ be an irreducible~$R$-representation of~$G$. 
\begin{enumerate}
\item Then~$\pi$ satisfies~$\mathbf{WGJFE}$. 
\item If, moreover, either
\begin{itemize}
\item  the cuspidal support of~$\pi$ does not contain the trivial character of~$\kk^\times$, 
\item or if~$0\neq \ell\mid q-1$, 
\end{itemize}
then~$\pi$ satisfies~$\mathbf{GJFE}$.  \end{enumerate}
\end{theo}

In a similar vein.

\begin{prop}\label{prop gamma unit}
Suppose that~$R$ is an algebraically closed field of characteristic different to~$p$.  Let~$(\pi,V)$ be an irreducible~$R$-representation of~$G$.
%\begin{enumerate}
%\item Then~$\pi$ satisfies (WGJFE).
%\item 
Suppose that~$\pi$ is a subquotient of~$\pi'\times \1^k$, where~$\pi'$ is a subquotient of a representation induced from cuspidal representations all different to~$\1$. Then 
\[\gamma(\pi,\psi)=(-q)^{-k/2}\g(\pi',\psi),\quad\text{and}\quad \g(\pi,\psi)\g(\pi^\vee,\psi^{-1})=q^{-k}.\]
%\end{enumerate}
\end{prop}
\begin{proof}
This follows from Corollary \ref{coro mult GJ}, Equation \eqref{eq dist gr} and Remark \ref{rem triv gamma}.
\end{proof}

\section{Gamma factors of distinguished cuspidal representations via the sign}\label{sec sym pair}

We now assume that~$R$ is a field\footnote{This is in addition to the assumption~$R$ is a~$\mathbb{Z}[\mu_p,|\Gg|^{-1/2}]$-algebra.} of characteristic different to~$p$.  We already applied the abstract setup of Section \ref{sec abstract setup} to obtain the Godement functional equation for cuspidal representations, by specifing to the symmetric pair corresponding to the so-called group case. 
Here we specialize $(G,H)$ to two other symmetric pairs, and deduce the value of gamma factors of absolutely irreducible cuspidal $H$-distinguished~$R$-representations.  In each case, we choose~$|\Gg|^{1/2}=|\Hh|$ in~$R$.

\subsection{The Galois model}

Here~$G=\GL_n(\k)$ with $n\geq 2$,~$\k$ is a quadratic extension of~$\k_0$, that~$\k_0$ has odd characteristic, 
and choose~$H=\GL_n(\k_0)$ with~$\Hh=\Mat_n(\k_0)$. With conventions from Section \ref{sec abstract setup}, this means that the role of~$\k$ there is played by~$\k_0$ here, i.e. we see~$\Gg$ as a semi-simple~$\k_0$ algebra. In particular we fix a non-trivial character~$\psi_0:\kk_0^\times \to R^\times$. The corresponding character~$\Psi$ on~$\Gg$ is thus~$\Psi=\psi_0\circ \tr_{\Gg/\k_0}=\psi_0\circ \tr_{\kk/\kk_0} \circ\tr_{\Gg/\k}$.

Let~$\pi$ be an absolutely irreducible cuspidal~$R$-representation of~$\GL_n(\k)$ distinguished by~$\GL_n(\k_0)$.  Then \ref{A1} ,\ref{A2} and \ref{A3} are satisfied.
% (and the existence of an~$H$-distinguished cuspidal irreducible
% representation implies that~$n$ is odd).
(For \ref{A2} in the modular case, see \cite{SecANT19} Remark 4.3.)

For \ref{A4},
note that~$ga=a$ is equivalent to~$(g-1)a=0$,
that is,
$g$ is of the form~$1+u$ where~$u$ is any matrix of~$\Gg$ whose kernel
contains the image of~$a$.

For \ref{A5}, we observe that~$\Hh^\bot=\delta \Hh$, where~$\delta\in \k^\times$ satisfies~$\tr_{\k/\k_0}(\delta)=0$. Note that~$\delta$ is unique up to~$\k_0^\times$-scaling. Thus, we have \ref{A5} with~$w=\delta I_n$. At this point, it is useful to remark that when one sees $G$ as a $k_0$-group, which is what we are doing to apply Proposition \ref{prop main}, then the gamma factor of $\pi$ with respect to $\psi_0$ is by definition equal to~$\gamma(\pi,\psi_0\circ \tr_{\k/k_0})$, when~$G$ is seen as a group over~$\k$. In this situation, Proposition \ref{prop main} takes the following form.

\begin{theo}\label{thm Gal gamma}
Let~$\pi$ be an absolutely irreducible cuspidal~$R$-representation of~$\GL_n(\k)$ distinguished by~$\GL_n(\k_0)$, and~$\psi_0:\kk_0\to R^\times$ a non-trivial character, and suppose that~$|M_n(\k)|^{1/2}=|M_n(\k_0)|$ in~$R$, then {\[\gamma(\pi,\psi_0\circ \tr_{\k/k_0})=\omega_{\pi}(\delta).\]}
\end{theo} 

Note that the a non-trivial character $\psi:\k\to R^\times$ is trivial on $\k_0$ if and only if it is of the form $(\psi_0\circ \tr_{\k/k_0})_{\delta}$, for $\psi_0:\kk_0\to R^\times$ a non-trivial character. Hence, in view of Equation \eqref{eq twisting psi in gamma}, we obtain the following.

\begin{coro}\label{cor triv gamma}
Let $n\geq 2$ and~$\pi$ be an absolutely irreducible cuspidal~$R$-representation of~$\GL_n(\k)$ distinguished by~$\GL_n(\k_0)$. If~$\psi$ is a non-trivial character of~$\k$ trivial on~$\k_0$, and~$|M_n(\k)|^{1/2}=|M_n(\k_0)|$  in~$R$, then \[\gamma(\pi,\psi)=1.\] 
\end{coro}
Notice that, changing the choice of square root (in the normalization of the Fourier transform) only changes one side of the functional equation and hence the gamma factor in that case is~$-1$.

\subsection{The linear and twisted linear models}\label{sec linear models}

Assume that~$n=2m$. Let~$G=\GL_{2m}(\k)$ and~$\Gg=\Mat_{2m}(\kk)$. Let~$\ll$ be a quadratic extension of~$\kk$, and embed~$\ll$ in~$\Gg$ as a~$\k$-sub-algebra; this embedding is unique up to~$G$-conjugacy. Let~$\pi$ be an absolutely irreducible cuspidal~$R$-representation of~$G$ distinguished by either by~$H=\GL_m(\k) \times \GL_m(\k)$, or by~$H=\GL_m(\ll)$.

First we choose~$H=\GL_m(\k) \times \GL_m(\k)$ block diagonally embedded inside~$G$, 
with~$\Hh=\Mat_m(\k) \times \Mat_m(\k)$. In this case, we assume that , and that~$\kk$ has odd characteristic. Then \ref{A2} follows from \cite[Corollary 2.16]{SecANT19}, and one checks that \ref{A3}, \ref{A4} and \ref{A5} are satisfied, with~$w={\rm antidiag}(1_m,1_m)$.  

Then choose~$H=\GL_m(\ll)$ with~$\Hh=\Mat_m(\ll)$. Conditions \ref{A3} and \ref{A4} are satisfied. We did not find a proof of \ref{A2} in the literature, hence we provide 
a proof of it in Appendix \ref{app A}.
For \ref{A5}, we write~$\ll=\kk[\sqrt{\a}]$ for some~$\a\in\k^\times$,~$\a\notin\k^{\times2}$, 
and embed it in~$\Mat_n(\kk)$ as
\begin{equation*}
\Mat_m(\ll) = \left\{
\begin{pmatrix} x & y \\ \a y & x \end{pmatrix}
\ \Big|\ x,y \in \Mat_m(\kk)\right\}.
\end{equation*}
Then~$\Hh^\bot = w\cdot\Hh$ with~$w={\rm diag}(1_m, -1_m)$. The specialization of Proposition \ref{prop main} now gives the following statement. 

\begin{theo}\label{thm lin 1}
Let~$\pi$ be an absolutely irreducible cuspidal~$R$-representation of~$\GL_{2m}(\k)$ distinguished either by~$H=\GL_m(\k) \times \GL_m(\k)$, or by~$H=\GL_m(\ll)$ for~$\ll/\kk$ a quadratic extension, embedded as above.  Moreover, in the linear case~$H=\GL_m(\k)\times\GL_m(\k)$, suppose that~$p$ is odd.  Let~$\psi$ be a non-trivial character of~$\kk$, and choose~$|\Gg|^{1/2}=|\Hh|$ in~$R$. Then \[\g(\pi,\psi)=\chi_{\pi}(w),\] where~$w={\rm antidiag}(1_m,1_m)$ if~$H=\GL_m(\k) \times \GL_m(\k)$, and~$w={\rm diag}(1_m, -1_m)$ if~$H=\GL_m(\ll)$. 
\end{theo} 
 
\section{Gamma factors and cuspidal parameters under reduction modulo $\ell$}

We now assume that $R$ is an algebraically closed field of characteristic different from $p$. In particular $\Hb(\kk,R^\times)=\Hom(\kk,R^\times)$, and we fix a non trivial character $\psi$ there. In this section $G$, $\Gg$, and $|\Gg|^{-1/2}$ are as in Section \ref{GJFEs}.

\subsection{Reduction modulo $\ell$}\label{sec red mod ell rep}
Assume that $R$ has positive characteristic $\ell$. We denote by $W_R$ its ring of Witt vectors, which is a complete discrete valuation ring, and by $F_R$ its fraction field. We fix $K=\overline{F_R}$ an algebraic closure of $F_R$, and by $\mathfrak{o}$ the integral closure of $W_R$ inside $K$. {Then by \cite[II,2]{Serre}, the valuation on $W_R$ extends uniquely to $\frak{o}$, and moreover $\o$ has again residue field $R$ since this latter is algebraically closed. This provides the reduction map $r_\ell:\frak{o}\to R$ extending the natural map from $W_R$ to $R$.}  We observe that $\overline{\mathbb{Q}}\subseteq K$. In the following discussion, extracted from \cite{V}, we use that finite dimensional representations of $G$ can always be realized over the algebraic closure of the prime subfield of their coefficient field.

A finite dimensional representation $(\pi_K,V_K)$ of $G$ with coefficients in $K$ is always integral: the action of $G$ stabilizes an $\mathfrak{o}$-lattice $V_e$ in $V_K$. Moreover, $V_e^\vee=\Hom_{\mathfrak{o}}(V_e,\mathfrak{o})$ is a $G$-stable $\mathfrak{o}$-lattice in $V_K$. Tensoring by $R$ gives rise to a representation $(\pi_R, V_R)$ of $G$ on $V_R=V_e\otimes_{\mathfrak{o}} R$, and $(\pi_R^\vee,V_R^\vee)$ naturally identifies with the representation of $G$ induced on $V_e^\vee\otimes_{\mathfrak{o}} R$. By the Brauer--Nesbitt principle, the semi-simplification $(r_\ell(\pi_K),r_\ell(V_K))$ of $(\pi_R, V_R)$ is independent of the choice of $V_e$. Here we make the standard abuse of notation consisting of using $r_\ell$ to denote both the reduction map $r_{\ell}:\mathfrak{o}\rightarrow R$, as well as for the map on representations defined above. 

Note that, the characters $\kk\to K^\times$ are~$\mathfrak{o}^\times$ valued, and there exists a unique character~$\psi_K:\kk\to K^\times$ lifting $\psi:\kk\rightarrow R^\times$.  Indeed, as in Section \ref{GJFEs}, both arise from a character~$\psi:\kk\rightarrow \mathbb{Z}[\mu_p]^\times$ via a~$\mathbb{Z}[\mu_p]$-algebra structure on~$\mathfrak{o}$.

Finally  we normalize the Fourier transform over $K$ by  choosing $|\Gg|^{-1/2}$ in~$K$ such that $r_\ell(|\Gg|^{-1/2})$ is our fixed choice of $|\Gg|^{-1/2}$ in $R$.

\subsection{Gamma factors and reduction modulo $\ell$}

Assume that $R$ has positive characteristic.

\begin{prop}\label{prop gamma mod ell}
Let $\pi_K$ be as in Section \ref{sec red mod ell rep}, and $\pi$ be a subquotient of $\pi_R$. Then if $\pi_K$ satisfies $\mathbf{WGJFE}$, resp $\mathbf{GJFE}$, then so does $\pi$. Moreover in this situation $\g(\pi_K,\psi_K)$ belongs to $\mathfrak{o}$ and 
\[\gamma(\pi,\psi)=r_\ell(\gamma(\pi_K,\psi_K)).\] 
\end{prop}
\begin{proof}
Let $(v,v^\vee)$ belong to $V_R\times V_R^\vee$, and choose $(v_e, v_e^\vee)$ in $V_e\times V_e^\vee$ lifting the pair, then the matrix coefficient $f_{v_e,v_e^\vee}\in \Cc(G,\mathfrak{o})$ reduces to $f_{v,v^\vee}\in \Cc(G,R)$. 
Similarly, for $\Phi\in \Cc(\Gg,R)$, resp. $\Cc(\Gg,R)$, there exists $\Phi_e\in \Cc(\Gg,{\mathfrak{o}})$, resp. $\Cc(G,\mathfrak{o})$, reducing to $\Phi$. By an appropriate choice of such data in the functional equation, we infer that $\g(\pi_K,\psi_K)$ belongs to $\mathfrak{o}$. Then, by a standard reduction modulo $\ell$ argument, we deduce that $\pi_R$ satisfies $\mathbf{WGJFE}$, resp $\mathbf{GJFE}$, and that $\gamma(\pi_R,\psi)=r_\ell(\gamma(\pi_K,\psi_K)).$ We finally conclude thanks to Lemma \ref{sqsgamma}.
\end{proof}

\subsection{The parametrization of cuspidal representations}\label{sec green param}

We set $G_n=\GL_n(\k)$. We recall that an irreducible~$R$-representation~$\pi$ of~$G_n$ is called \emph{supercuspidal} if it does not occur as a subquotient of a proper parabolic induction. In particular supercuspidal representations are cuspidal, and the converse holds when~$R$ has characteristic zero. In this section, we summarize the results of \cite{Green}, \cite{James} and \cite{DJ}, and refer to \cite{MScont} for a modern exposition. They are stated in these references for~$R=\mathbb{C}$ and~$R=\overline{\mathbb{F}_\ell}$ for~$\ell$ a prime number not dividing~$q$, but they hold as stated below since irreducible~$R$-representations of finite groups can always be realized over a finite extension of the prime subfield of~$R$. 

We denote by~$\kk_n$ the unique up to isomorphism extension of~$\kk$ of degree~$n$, and we once and for all embed it in~$\Gg_n=\Mat_n(\kk)$, recalling that such an embedding is unique up to~$\GL_n(\kk)$-conjugacy. We continue with the notation~$G_n=\GL_n(\kk)$.

A character~$\xi:\k_n^\times\rightarrow R^\times$ is called \emph{regular} if its orbit \[\mathcal{O}(\xi)=\{\xi^\tau:\tau\in\Gal(\k_n/\k)\}\] under the Galois group ~$\Gal(\k_n/\k)$ has maximal size~$n$. By \cite{Green}, when~$R$ has characteristic zero, there is a surjective map
\begin{align*}
\{\text{regular characters from}\ \k_n^\times \ \text{to} \ R^\times\}&\rightarrow \{\text{cuspidal~$R$-representations of }G_n\}/\simeq\\
\xi&\mapsto \pi(\xi),
\end{align*}
The character formula given in \cite{Green} also shows the following:
\begin{enumerate}
\item Two such cuspidal representations~$\pi(\xi)$ and~$\pi(\xi')$ are isomorphic if and only if there exists $\tau\in\Gal(\k_n/\kk)$ such that~$\xi'=\xi^\tau$, 
\item The dual~$\pi(\xi)^\vee$ is isomorphic to~$\pi(\xi^{-1})$,
\item If~$\kk/\kk'$ is a Galois extension, $\tau\in \Gal(\kk/\kk')$ and $\tau_n\in \Gal(\k_n/\k')$ extends $\tau$, we have~$\rho(\xi^{\tau_n})\simeq\rho(\xi)^\tau$. 
\end{enumerate}  

We now recall the classification of cuspidal~$R$-representations of James \cite{James} when~$R$ has positive characteristic~$\ell\neq p$.

Let~$\pi_K$ be an irreducible cuspidal~$K$-representation of~$G_n$, then it is~$\mathfrak{o}$-integral, and moreover~$r_{\ell}(\pi_K)$ is irreducible and cuspidal.  A character~$\xi:\kk_n^\times\rightarrow K^\times$ uniquely decomposes as~$\xi=\xi_\ell\xi^\ell$, with~$\xi_\ell$ of order prime to~$\ell$ and~$\xi^\ell$ of order a power of~$\ell$. We say that~$\xi$ is \textit{$\ell$-regular} if~$\xi=\xi_\ell$. For~$\xi\in \Hom(\k^\times,K^\times)$, its image lies in~$\mathfrak{o}^\times$, and we set~$r_\ell(\xi)=r_{\ell}\circ \xi$. It then follows from \cite{James} that the map  

\begin{align*}
r_{\ell}:\{\text{cuspidal~$K$-representations of }G_n\}/\simeq&\rightarrow \{\text{cuspidal~$R$-representations of }G_n\}/\simeq\\
\pi_K &\mapsto r_{\ell}(\pi_K)
\end{align*} 
is surjective, giving rise, in view of the Green parametrization, to a surjection 
\begin{align*}
\{\text{regular characters from}\ \k_n^\times \ \text{to} \ K^\times\}&\rightarrow \{\text{cuspidal~$R$-representations of }G_n\}/\simeq\\
\xi&\mapsto r_{\ell}(\pi(\xi)).
\end{align*}

The parametrization of James enjoys the following properties:
\begin{enumerate}
\item For regular~$\xi$ and~$\xi'$,~$r_{\ell}(\pi(\xi'))\simeq r_{\ell}(\pi(\xi))$ if and only if there exists~$\tau\in\Gal(\kk_n/\kk)$ such that~$\xi'_\ell=\xi_\ell^\tau$, and we observe that this is the same as saying that there exists~$\tau\in\Gal(\ll/\kk)$ such that~$r_\ell(\xi')=r_\ell(\xi)^\tau$.  
\item~$r_{\ell}(\pi(\xi))$ is supercuspidal if and only if~$\xi_\ell$ is regular, or equivalently~$r_\ell(\xi)$ is regular.
\end{enumerate}

We set \[\pi(r_\ell(\xi)):=r_{\ell}(\pi(\xi)),\] for any regular character~$\xi:\kk_n^\times\rightarrow K^\times$. 

Hence the irreducible cuspidal~$R$-representations of~$G$ are naturally parametrized by (Galois orbits of) elements in~$\Hom(\k_n^\times,R^\times)$: when~$R$ 
has characteristic zero, it is the set of (Galois orbits of) regular characters whereas it is the set of characters with a regular~$\ell$-adic lift when~$R$ has positive characteristic. In both cases supercuspidal representations are parametrized by (Galois orbits of) regular characters. 

For~$\rho$ a supercuspidal representation of~$G_k$, and~$r\in \mathbb{N}-\{0\}$, we denote by~$\st_r(\rho)$ the unique generic subquotient of the parabolically induced representation~$\underbrace{\rho\times \dots \times \rho}_{r  \times}$ (see \cite[Chapitre 3, 2]{V}).  From the multiplicativity relation for gamma factors we have:
\begin{lemm}\label{strrhomult}
Let~$\rho$ a supercuspidal representation of~$G_k$, then
\[\gamma(\st_r(\rho),\psi)=\gamma(\rho,\psi)^r.\]
\end{lemm}

We now recall from \cite{DJ} the parametrization of cuspidal representations of~$G$ in terms of supercuspidal representations when $R$ has positive characteristic, and relate it to the Green-James parameters as explained in \cite{V}.

For~$f$ a positive integer dividing~$n$, we set~$r=n/f$, let~$o(q^f)$ denote the order of~$q^f$ in~$\k^\times$ and~
\begin{equation*}
e(q^f):=\begin{cases}
o(q^f)&\text{if~$o(q^f)\neq 1$};\\
\ell&\text{otherwise.}
\end{cases}
\end{equation*}
%
%$e(q^f)$ to be the order of~$q^f$ in~$k^\times$ if the latter is not equal to one, and~$e(q^f)=\ell$ otherwise. 
\begin{prop}
	Assume that $R$ has positive characteristic.
\begin{enumerate}
\item Let~$\rho$ be an irreducible supercuspidal~$R$-representation of~$G_f=\GL_f(\kk)$.  If~$r=e(q^f)\ell^a$ for~$a\in \mathbb{N}$, then~$\st_r(\rho)$ is cuspidal.
\item Conversely, if~$\pi$ is an irreducible cuspidal~$R$-representation of~$G_n$, there exists a unique~$r$ equal to one (supercuspidal case) or~$e(q^f)\ell^a$ as above, and a unique supercuspidal representation~$\rho$ of~$G_f=\GL_f(\kk)$ such that~$\pi=\st_r(\rho)$.
\item Let~$\st_r(\rho)$ be an irreducible cuspidal~$R$-representation of~$G_n$.  
\begin{enumerate}
\item For an automorphism~$\tau$ of~$\kk$, we have~$\st_r(\rho)^\tau\simeq \st_r(\rho^\tau)$.
\item We have~$\st_r(\rho)^{\vee}\simeq \st_r(\rho^\vee)$.
\end{enumerate}
In particular,~$\st_r(\rho)$ is self-dual (resp.~$\sigma$-self-dual in the Galois setting) if and only if~$\rho$ is self-dual (resp.~$\sigma$-self-dual).
\end{enumerate}
\end{prop}

Moreover, it follows from the proof \cite[Chapitre 2, 2.8, Th\'eor\`eme, a)]{V} that if~$\pi=\pi(\xi_0)$, then~$f=|\mathcal{O}(\xi_0)|,$ and if one defines~$\xi_0':\kk_f^\times \to k^\times$ by the relation 
\[\xi_0=\xi_0'\circ N_{\kk_n/\kk_f},\] then~$\xi_0'$ is regular and~$\pi(\xi_0)\simeq\st_r(\pi(\xi_0'))$.
\\
\section{Gamma factors of self-dual and Galois self-dual cuspidal representations}\label{computation}%, 

We continue with $R$ an algebraically closed field of characteristic different from $p$, and we return to the notation~$G_n=\GL_n(\k)$. When $R$ has positive characteristic, we define $K$ and $\mathfrak{o}$ as in Section \ref{sec red mod ell rep}. As before we choose square roots so that~$|\Gg_n|^{1/2}=|\Hh_n|$ in $R$ (and $K$ when $R$ has positive characteristic). 
In this section, we compute the gamma factors of irreducible cuspidal self-dual~$R$-representation of~$\GL_{2m}(\k)$, and irreducible cuspidal~$\sigma$-self-dual~$R$-representations of~$\GL_{n}(\k)$ (where~$\k/\k_0$ is quadratic) in terms of the Green and Dipper--James parametrization of cuspidal representations.  We reduce to this computation to the supercuspidal setting, where noting self-dual and~$\sigma$-self dual supercuspidals are distinguished, we give two methods: one based on the results of Section \ref{sec sym pair}, and the other based on Kondo's formula (Lemma \ref{Kondo}). 

\subsection{Results on distinguished cuspidal representations}

We have the following result:
\begin{prop}[{\cite[Lemmas 2.5, 2.19]{SecANT19}, Corollary \ref{cor glfd sefdual tw lin}, Proposition \ref{cor glfd sefdual tw lin}, \cite[Corollary 5.4]{MR4089564}}]\label{distsdssd}
Let~$\pi$ be an irreducible cuspidal~$R$-representation of~$G_n$.
\begin{enumerate}
	\item (Galois case) $\k$ is a quadratic extension of~$\k_0$,~$H_n=\GL_n(\k_0)$, and~$\Gal(\k/\k_0)=\langle \sigma\rangle$.  \begin{enumerate}
		\item If $\pi$ is $H_n$-distinguished, then  it is~$\sigma$-self-dual.  
		\item If moreover~$\pi$ is supercuspidal, then the converse implication also holds.
	\end{enumerate}
	\item (Linear/twisted-linear case)~$n=2m$,~$H_n=\GL_m(\k)\times\GL_m(\k)$ (embedded diagonally)-- in which case we suppose that~$p$ is odd, or~$H_n=\GL_m(\ll)$ for~$\ll/\k$ quadratic.
	\begin{enumerate}
		\item If $\pi$ is $H_n$-distinguished, then  it is~self-dual.  
		\item \label{finaldistsdssd} If moreover~$\pi$ is supercuspidal, then the converse implication also holds.
	\end{enumerate}
\end{enumerate}
\end{prop}
\begin{proof}
All statements follow directly from the references, except the final statement \eqref{finaldistsdssd} for the twisted-linear case.  The reference \cite[Corollary 5.4]{MR4089564} proves this when $R$ has characteristic zero. When $R$ has positive characteristic, it follows from \cite{SecANT19} that under reduction modulo~$\ell$, self-duality lifts for supercuspidal representations, whereas distinction always descends, so we deduce statement \eqref{finaldistsdssd} in the positive characteristic case.
\end{proof}

\subsection{Explicit formula for gamma factors of irreducible distinguished supercuspidal~$R$-representations using the sign}\label{sec dist gamma 2}
Suppose that~$\kk=\k_0$ has odd characteristic. Let~$\pi(\xi)$ be an irreducible self-dual cuspidal~$R$-representation of~$G_n=\GL_n(\kk)$, with~$n\geq 2$. We are going to compute~$\g(\pi,\psi)$ in terms of~$\xi$, using distinction of the supercuspidal support of~$\pi$. % in the case of linear and twisted linear models. 
We first recall the following results, which follows from \cite[(2.7)]{SecANT19}.

\begin{prop}\label{prop supercusp dist}
Assume that~$q$ is odd, and let~$\pi(\xi)$ be an irreducible supercuspidal~$R$-representation of~$G_n$. 
\begin{enumerate}
\item Suppose~$n\geq 2$, then~$\pi(\xi)$ is self-dual if and only if~$n=2m$ is even and~$\xi$ restricts trivially to~$\k_m^\times$.
\item If~$n=1$,~$\pi=\xi$ is self-dual if and only if~$\xi^2=1$.
\end{enumerate}% if and only if~$n=2m$ is even and~$\pi(\xi)$ is~$\GL_m(\k)\times \GL_m(\k)$-distinguished.
\end{prop}

%We are now in a position to prove the main result of this section. 
We denote by~$\d$ any element in~$\k_n^\times$ such that~$\tr_{\k_n/\k_m}(\delta)=0$, or equivalently such that~$\delta \k_m^\times$ is the unique order two element in~$\k_n^\times/\k_m^\times$.

\begin{prop}\label{sccomp}
Assume that~$q$ is odd and~$n=2m\geq 2$. Let~$\pi=\pi(\xi)$ be an irreducible self-dual supercuspidal~$R$-representation of~$G_n$.  Then
\[\g(\pi,\psi) =-\xi(\d).\]
\end{prop}

\begin{proof}
If~$R$ has characteristic zero, the result follows from Proposition \ref{prop supercusp dist} and \cite[Lemme 3.4.12, Lemme 4.3.13]{Coniglio}. When~$R$ has positive characteristic and~$\pi$ is supercuspidal, the result follows from the characteristic zero case and Proposition \ref{prop gamma mod ell}, as a self-dual supercuspidal~$R$-representation has a self-dual supercuspidal lift.
\end{proof}

\begin{theo}\label{thm dist gamma green}
Assume that~$q$ is odd and~$n=2m\geq 2$. Let~$\pi=\pi(\xi)$ be an irreducible self-dual cuspidal~$R$-representation of~$G$. Write~$\pi$ under the form~$\pi=\st_r(\rho)$ for an irreducible supercuspidal~$R$-representation~$\rho=\pi(\xi')$, i.e.~$\xi=\xi'\circ N_{\k_n/\k_f}$ with~$fr=n$ and~$\xi'$ regular. Then:
%\begin{enumerate}
%\item If~$\pi$ admits a linear or  twisted linear model, as in Section \ref{sec linear models}, then~$\pi$ is self-dual.
%\item Assume that~$\pi$ is self-dual.
\begin{enumerate}
\item If~$f\geq 2$, then~$\g(\pi,\psi) = (-\xi'(\d))^r$, i.e.,~
\begin{align*}
\g(\pi,\psi)= \begin{cases}-\xi(\d)&\text{ when~$r$ odd;}\\
1&\text{ when~$r$ is even. }\end{cases}\end{align*}
\item If~$f=1$, either~$\pi\simeq \st_n(\mathbf{1})$, and we have
\[\gamma(\st_n(\mathbf{1}),\psi)=q^{-m};\]
or~$\ell\neq 2$,~$\pi\simeq \st_n(\xi)$ for the non-trivial quadratic character~$\xi:\kk^\times \rightarrow R^\times$, and
\[\gamma(\st_n(\xi),\psi)=\xi(-1)^m.\]
%\begin{align*}
%g(\pi,\psi)=\begin{cases}
%q^m&\text{if~$\pi
% \xi'(-1)^m
%\end{cases}
%\end{align*}
%\end{enumerate} 
\end{enumerate}
\end{theo}
\begin{proof}
%The first assertion follows from Proposition \ref{prop cusp lin sd} below. 
%For the second, we first assume that~$f\geq 2$. If~$R$ has characteristic zero, the result follows from Proposition \ref{prop supercusp dist} and \cite[Lemme 3.4.12, Lemme 4.3.13]{Coniglio}. When~$R$ has positive characteristic and~$\pi$ is supercuspidal, the result follows from the characteristic zero case and Equation \eqref{eq red gamma mod ell}. 
Follows directly from Lemma \ref{strrhomult}, Propositions \ref{prop supercusp dist} and \ref{sccomp}, and Remark \ref{rem triv gamma}.
%When~$r\geq 2$, it follows from the multiplicativity relation of Corollary \ref{coro mult GJ} and the supercuspidal case.  
%When~$f=1$, we can still apply multiplicativity to obtain that 
%$\g(\pi,\psi) =\gamma(\xi',\psi)^{2m}$. Now~$\xi'=\xi'^{-1}$ by self-duality of~$\pi$. We thus obtain the equality 
%$\g(\xi',\psi)^2=\xi'(-1)$ as a consequence of Proposition \ref{prop gamma unit}. \rob{Not any more when $1$ in supercuspidal support}
\end{proof}

\begin{rema}\label{rem gal green}
Assume now that~$\k/\k_0$ is a quadratic extension, and that~$\delta_0\in \k^\times$ is a trace zero element. If~$R$ has characteristic zero, it follows from the Green character formula evaluated at central elements, that~$\omega_{\pi}(\delta_0)=\xi(\delta_0)$, where~$\xi$ is the Green parameter of~$\pi$, and by reduction modulo~$\ell$, the same holds with~$\xi$ the James parameter of~$\pi$ in positive characteristic. Fix a non-trivial character~$\psi_0:\k_0\to R^\times$. If~$\pi$ is~$\GL_n(\k_0)$-distinguished, Theorem \ref{thm Gal gamma} then gives the formula 
\[\gamma(\pi(\xi),\psi_0\circ \tr_{\k/\k_0})=\xi(\delta_0).\] 
\end{rema}

\subsection{An alternative approach using Kondo's formula}\label{sec dist gamma 2}
In the previous section our first step was to reduce the computation of gamma factors of self-dual/$\sigma$-self-dual cuspidal representations to the supercuspidal case. Here we propose another approach in the supercuspidal case, based on the work of Kondo.  

\begin{lemm}[Kondo's formula]\label{Kondo}
Let~$\pi=\pi(\xi)$ be an irreducible cuspidal~$R$-representation of~$G_n$.  Then
\[\gamma(\pi(\xi),\psi)=(-1)^{n-1}\gamma(\xi, \psi ).\]
\end{lemm}
\begin{proof}
For~$R$ of characteristic zero this is proved by Kondo in \cite{Kondo}.  The modular case follows by reduction modulo~$\ell$, using Proposition \ref{prop gamma mod ell}.
\end{proof}

%In characteristic zero, for~$\xi:\k_n^\times \to R^\times$ a regular character, Kondo proved in \cite{Kondo} the following formula: 
%
%\[\gamma(\pi(\xi),\psi)=(-1)^{n-1}\gamma(\xi,\psi).\]

We recover a different proof of Proposition \ref{sccomp} % as well as Remark \ref{rem gal green}, 
from the following computation.

\begin{lemm}
Assume that~$q$ is odd and that~$n=2m$ is even. Let~$\xi:\k_n^\times\to R^\times$ be a non-trivial character such that~$\xi_{|\k_m^\times}=\bf{1}$. Let~$\d\in \k_n^\times$ be such that~$\tr_{\k_n/\k_m}(\delta)=0$, and~$\psi:\k_m\to R^\times$ be a non trivial character. Then 
\[\gamma(\xi,\psi)=\xi(\delta).\]  
\end{lemm}
\begin{proof}
We have
\[q^m\gamma(\xi,\psi)=\sum\limits_{x\in \k_n^\times} \xi^{-1}(x)\psi(\tr_{\k_n/\k_m}(x))=\sum\limits_{x\in \k_m^\times\backslash \k_n^\times} \xi^{-1}(x)\sum\limits_{t\in \k_m^\times} \psi(t \tr_{\k_n/\k_m}(x)).\] If~$x\notin \delta^{-1} k_m^\times$, the character~$t\in \k_m\to \psi(t \tr_{\k_n/\k_m}(x))$ is non trivial and \[\sum\limits_{t\in \k_m^\times} \psi(t \tr_{\k_n/\k_m}(x))=-1.\] On the other hand if~$x=\delta^{-1} t_0$ with~$t_0\in \k_m^\times$, then 
\[\sum\limits_{t\in \k_m^\times} \psi(t \tr_{\k_n/\k_m}(x))=q^m-1.\] All in all:
\[q^m\gamma(\xi,\psi)=-\sum\limits_{x\in \k_m^\times\backslash \k_n^\times-\delta^{-1} \k_m^\times}\xi^{-1}(x)+\xi(\delta)(q^m-1)=\xi(\delta)- \sum\limits_{x\in \k_m^\times\backslash \k_n^\times}\xi^{-1}(x)+\xi(\delta)(q^m-1)\] \[\ =\xi(\delta)+ 0+\xi(\delta)(q^m-1)=q^m \xi(\delta).\]
\end{proof}

\begin{rema}
By Proposition \ref{distsdssd}, for irreducible supercuspidal~$R$-representations we have completed our quest, computing both the gamma factor and the sign of the distinguishing linear form (which are equal by Section \ref{sec sym pair}).  For irreducible cuspidal non-supercuspidal~$R$-representations, this still leaves the task of identifying which self-dual and~$\sigma$-self-dual representations are distinguished. We achieve this in 
\cite{KMSII}.
%
%We consider this question in future work.  For now, notice that~$st_n(1)$ is not distinguished if~$q^m\neq 1[\ell]$. 
\end{rema}

\section{Gamma factors of irreducible distinguished representations}
In this section, we follow the notation of Section \ref{sec sym pair}, in particular we set~$G=\GL_n(\k)$.  Because~gamma factors of irreducible~$R$-representations of~$G$ only depend on their cuspidal support, one can deduce a formula for~gamma factors of irreducible distinguished~$R$-representations of~$G$ from Mackey theory. 
Here we denote by~$\k_0$ the field~$\k$ if we are interested in the linear or twisted linear models, and by~$\k_0$ a subfield of~$\k$ with~$[\k:\k_0]=2$ if we consider the Galois model. For~$\pi$ a representation of~$G$, we set 
$\pi^*=\pi^\vee$ in the linear cases, whereas we set~$\pi^*=(\pi^\vee)\circ \sigma$ in the Galois case, where~$\sigma$ is the Galois involution attached to~$\k/\k_0$. 

\begin{prop}\label{prop the Matringe method}
Assume that~$q$ is odd in the case of linear models. Let~$\rho_1,\dots,\rho_k$ be cuspidal irreducible representations of~$\GL_{n_1}(\k), \dots,  \GL_{n_k}(\k)$ respectively, and assume that~$\rho_1\times \dots \times \rho_k$ is distinguished. Then there exists~$\iota$ an involution in~$S_k$ such that~$\rho_{\iota(i)}=\rho_i^*$ for each~$i=1,\dots,k$, and~$\rho_i$ is distinguished if~$\iota(i)=i$. 
\end{prop}
\begin{proof}
The method of \cite[Lemma 5.13]{KMS} applies verbatim. 
\end{proof}

Let~$\pi$ be an irreducible representation. We can realize it as a quotient of~$\rho_1\times \dots \times \rho_k$ as above, where the set~$\{\rho_1, \dots,\rho_k\}$ is unique. If~$\pi$ is moreover distinguished and $q$ is assumed to be odd in the linear case, then $\pi$ is self-dual, and we can choose a re-ordering of the cuspidal support in the form 
$\{\rho_1,\dots,\rho_a, \tau_1,\tau_1^*,\dots,\tau_b,\tau_b^*,c\mathbf{1}\},$ with each~$\rho_i$ distinguished but non trivial, and each $\tau_i$ is non trivial. Observe that in the case of linear or twisted linear models, the multiplicity $c$ of $\mathbf{1}$ in the cuspidal support has to be even. 

\begin{coro}
Let~$\pi$ be a distinguished representation of~$G$, and let the~$\rho_i,\tau_j$ be as above. Let~$\psi$ to be a non-trivial character of~$\k$, assumed to be trivial on~$\k_0$ if we consider Galois models.
\begin{enumerate}
	\item  If~$\pi$ has a linear or twisted linear model, and $q$ is moreover odd in the linear case, then \[\gamma(\pi,\psi)=(-q^{-1/2})^c\prod_{i=1}^a\gamma(\rho_i,\psi)\prod_{i=1}^b\omega_{\tau_i}(-1),\] where the~$\gamma(\rho_i,\psi)$ are computed by Theorem \ref{thm dist gamma green} when~$\dim(\rho_i)\geq 2$, and are quadratic Gauss sums otherwise (when $q$ is even, there is no one dimensional $\rho_i$).
	\item If~$\pi$ has a Galois model, then  \[\g(\pi,\psi)=(-q^{-1/2})^c.\]
\end{enumerate} 
\end{coro}
\begin{proof}
From multiplicativity, Proposition \ref{prop gamma unit}, Proposition \ref{prop the Matringe method}, and Remark \ref{rem triv gamma} we immediately get the formula~$\gamma(\pi,\psi)=\prod_{i=1}^a\gamma(\rho_i,\psi)\prod_{i=1}^b\omega_{\tau_i}(-1(-q^{-1/2})^c)$ in the linear cases. In the Galois case, by the same argument, we obtain~$\gamma(\pi,\psi)=\prod_{i=1}^a\gamma(\rho_i,\psi)(-q^{-1/2})^c$, and the result follow from Corollary \ref{cor triv gamma}. 
\end{proof}

\appendix

\sectionApp{Multiplicity one for twisted linear models}\label{app A} 

As promised, we verify condition \ref{A2} in the case of twisted linear models. We actually check that~$(G,H)$ is a Gelfand pair in the case of linear models, which in particular implies self-duality of irreducible representations, as follows from the results of \cite{Aizen} and \cite{Zhang}. So~$n=2m$,~$G=\GL_{2m}(\kk)$ and~$H\simeq \GL_m(\ll)$ as in Section \ref{sec linear models}. We claim that:

\begin{prop}\label{prop dble classes}
The double cosets in~$H\backslash G/H$ are fixed by~$g\to g^{-1}$. 
\end{prop}
\begin{proof}
For~$p$-adic fields this is proved in \cite{Guo} extended in \cite{BM} to non-archimedean local fields of characteristic different to two. The proof of this fact holds in the setting of finite fields, and we quickly recall why. 
\begin{enumerate}
\item First \cite[Lemma 3.4]{Guo} holds in the finite field setting with the same proof, and any~$s$ in the symmetric space is~$H$-conjugate to an element 
$\sigma(s_1,s_2,s_3)$ as in the notations of \cite{Guo}.
\item As in the proof of \cite[Proposition 3]{Guo}, this reduces the proof of the triviality of the action of~$\iota$ on~$H\backslash G/H$ to the proof of \cite[Lemmas 3.1, 3.2 (2) and (3), and 3.3]{Guo}.  
\item Then \cite[Lemma 3.1]{Guo} holds with the same proof. 
\item Then \cite[Lemma 3.2 (2) and (3)]{Guo} holds with a different proof, which is given in \cite[Lemma 4.2]{BM}. Note that \cite[Lemma 4.2]{BM} only proves \cite[Lemma 3.2 (2)]{Guo} but then \cite[Lemma 3.2 (3)]{Guo} holds in the finite setting with the same proof since~$q\neq 2$. 
\item Finally \cite[Lemma 3.3 second statement]{Guo} holds with a different proof, given in \cite[Lemma 4.4]{BM}, whereas \cite[Lemma 3.3 first statement]{Guo} follows at once from what has already been proved in the adaptation of \cite[Lemma 3.2]{Guo} to our finite setting. 
\end{enumerate}
\end{proof}

As a corollary, we obtain the following statement.

\begin{coro}\label{cor glfd sefdual tw lin}
The pair~$(G,H)$ is a Gelfand pair, i.e. if~$\pi$ is an~$H$-distinguished representation of~$G$, then~$\Hom_{R[H]}(\pi,R)$ has dimension one, and moreover~$\pi\simeq \pi^\vee$. 
\end{coro}
\begin{proof}
As observed in \cite[Proof of Theorem 4.1 and Remark 4.3]{SecANT19}, the argument of Prasad in \cite{Pr1990} applies as soon as the anti involution $x\mapsto \sigma(x)^{-1}$ fixes $H\backslash G/H$ pointwise. But this latter fact indeed follows from Proposition \ref{prop dble classes}. 
\end{proof}

\sectionApp{Self-duality for cuspidal modular linear models}\label{app B} 

In this paragraph,~$\k$ is a field of characteristic different to two, the integer~$n=2m$ is even, and 
$H\simeq \GL_m(\k)\times\GL_m(\k)$ as in Section \ref{sec linear models}. We observe that this pair is not a Gelfand pair, although it is a cuspidal Gelfand pair as already explained. This follows for example from \cite[Theorem 2]{Aizen}, or from observing that the Steinberg representation of~$\GL_2(\k)$ admits two non-collinear~$H$-invariant linear forms. However, when~$R$ has characteristic zero, it follows from \cite[Theorem A]{Kapon} that whenever an irreducible representation of~$G$ is distinguished, it is self-dual. It is very likely that this result hold in this generality in the modular case, using the techniques of Kapon to reduce to the cuspidal case that we prove here. However, we only need to know that~$H$-distinguished cuspidal representations are self-dual to prove Theorem \ref{thm dist gamma green}, so we content ourselves with this case:

\begin{prop}\label{prop cusp lin sd}
Assume that~$q$ is odd, and~$\pi$ is an~$H$-distinguished cuspidal irreducible~$R$-representation of~$G$, then~$\pi\simeq \pi^\vee$.
\end{prop}
\begin{proof}
By a straightforward adaptation of \cite[Lemma 5.8]{KMS}, the representation~$\pi$ occurs as a component of~$r_{\ell}(\pi_{\ell})$  (see Section \ref{sec green param}) for~$\pi_{\ell}$ an irreducible representation of~$G$ which is~$H$-distinguished. Now, by \cite[Theorem A]{Kapon}, the representation~$\pi_{\ell}$ self-dual. 
Realizing~$\pi_{\ell}$ as a subrepresentation of~$\rho_1\times \dots \times \rho_t$ with~$\rho_i$ cuspidal, we deduce from \cite[Lemma 5.13]{KMS} that there exists an involution~$\iota$ in~$S_t$ such that~$\rho_{\iota(i)}\simeq \rho_i^\vee$ for all~$i=1,\dots,t$. Now note that if we write~$\pi=\st_r(\rho)$ for~$\rho$ supercuspidal, then following \cite[5.9]{KMS}, for each~$i=1,\dots,r$ we have~$r_{\ell}(\rho_i)=\st_{r_i}(\rho)$ for an appropriate choice of~$r_i$. The relation~$\rho_{\iota(1)}\simeq \rho_1^\vee$ then implies that 
$\st_{\iota(r_1)}(\rho)\simeq \st_{r_i}(\rho)^\vee$, hence that~$\iota(r_1)=r_1$ and that~$\rho$ is self-dual. In particular~$\pi$ is self-dual. 
\end{proof}

\bibliographystyle{plain}
\bibliography{GJ}

\end{document}